\documentclass[11pt]{article}
\usepackage[cp1251]{inputenc}
\usepackage[english]{babel}
\usepackage{amsmath}
\usepackage{afterpage}
\usepackage{amsfonts}
\usepackage{mathrsfs}
\usepackage{amssymb}
\usepackage{epsfig}
\usepackage{graphicx}
\usepackage{multirow}
\usepackage{hanging}
\usepackage{nccmath}
\usepackage[flushmargin, symbol*]{footmisc}
\usepackage[usenames,dvipsnames]{color}
\newtheorem{theorem}{\textsf{Theorem}}
\newtheorem{lemma}{\textsf{Lemma}}

\newenvironment{proof}[1][\textsf{Proof. }]{\textbf{#1}}{$\square$}

\newenvironment{example}[1][\textsf{Example. }]{\textbf{#1}}{$\lozenge$}

\newtheorem{conj}{\textsf{Conjecture}}

\begin{document}

\title{
\date{}
{
\large \textsf{\textbf{Combinatorial decompositions, Kirillov-Reshetikhin invariants and the Volume Conjecture for hyperbolic polyhedra}}
}
}
\author{\small Alexander Kolpakov \hspace*{42pt}
\small Jun Murakami}
\maketitle

\begin{abstract}\noindent

We suggest a method of computing volume for a simple polytope $P$ in three-dimensional hyperbolic space $\mathbb{H}^3$. This method combines the combinatorial reduction of $P$ as a trivalent graph $\Gamma$ (the $1$-skeleton of $P$) by $I-H$, or Whitehead, moves (together with shrinking of triangular faces) aligned with its geometric splitting into generalised tetrahedra. With each decomposition (under some conditions) we associate a potential function $\Phi$ such that the volume of $P$ can be expressed through a critical values of $\Phi$. The results of our numeric experiments with this method suggest that one may associated the above mentioned sequence of combinatorial moves with the sequence of moves required for computing the Kirillov-Reshetikhin invariants of the trivalent graph $\Gamma$. Then the corresponding geometric decomposition of $P$ might be used in order to establish a link between the volume of $P$ and the asymptotic behaviour of the Kirillov-Reshetikhin invariants of $\Gamma$, which is colloquially know as the Volume Conjecture. 

\end{abstract}

\parindent=0pt

\section{Introduction}\label{section:introduction}

In this paper we suggest a method of computing volume for a polytope $P$ in three-dimensional hyperbolic space $\mathbb{H}^3$. We require the polytope $P$ to be simple, so that the computation of volumes may combine two processes of initially different nature. First, a reduction of $P$ (as a trivalent graph) to the tetrahedron by a sequence of $I-H$ (or Whitehead) and so-called ``capping'' moves. And second, a decomposition of $P$ into a number of generalised tetrahedra $T_i$, $i=1,\dots,n$, such that $T_i$ partition $P$ into geometric parts, and $\mathrm{Vol}\,P = \sum^n_{i=1} \mathrm{Vol}\,T_i$. 

With such a decomposition we associate a function $\Phi(\ell_1,\dots,\ell_m)$, $m\geq 1$, that depends on some additional geometric parameters of the decomposition (in fact. lengths of some common perpendiculars to the faces of $P$). Under some conditions, we show that $\mathrm{Vol}\,P$ can be expressed through the value of $\Phi(\ell^\ast_1,\dots,\ell^\ast_m)$ at a critical point $(\ell^\ast_1,\dots,\ell^\ast_m)$. 

Finally, we use our method in a large number of examples, by implementing it in Wolfram Mathematica\textsuperscript{\textregistered} \cite{W}. We also check our computations, where possible, with the Orb software \cite{Orb}. 

This part of our work falls in line with the recent study on the Volume Conjecture for generalised tetrahedra by Costantino and Murakami \cite{CM} and for general hyperbolic polytopes by Costantino, Gu\'{e}ritaud and van der Veen \cite{CGV}. 

Next,given a polytope $P$, we consider its $1$-skeleton $\Gamma$ as a trivalent graph, compute the Kirillov-Reshetikhin invariants of $\Gamma$ with an appropriate colouring of its edges determined by the corresponding dihedral angles of $P$ and study their asymptotic behaviour. We conjecture that the sequence of combinatorial moves that is used in order to compute the Kirillov-Reshetikhin invariants of the corresponding trivalent graph $\Gamma$ is associated with the sequence of moves used in the combinatorial reduction of $P$. Then the geometric decomposition of $P$ associated with the respective sequence of $I-H$ and capping moves might be used in order to establish a link between the volume of $P$ and the asymptotic behaviour of the Kirillov-Reshetikhin invariants of $\Gamma$. Such a link has been colloquially named the Volume Conjecture by various authors, and is first established in \cite{Kashaev}.

The paper is structured as follows: after recalling some preliminary results on the volume formulas for generalised hyperbolic tetrahedra (cf. \cite{KM2012, KM-err} and \cite{KM2014}), we proceed to the description of our method and formulate the main statement of the paper. We illustrate our method in two main working examples: computing the volume of a hyperbolic prism $\Pi$, and that of a hyperbolic ``pleated'' prism $\mathscr{P}$. Then, we provide a more computationally complicated example of a dodecahedron with various Coxeter dihedral angles. However, the method of present paper is not restricted to Coxeter polytopes. We expect it to work in a wide variety of dihedral angles, and expect that it can be generalised to allow computing volumes of knotted trivalent graphs. 

Finally, we state a number of conjectures relating the volume of a hyperbolic polyhedron $P$ to the Kirillov-Reshetikhin invariants of its $1$-skeleton $\Gamma$, viewed as a trivalent graph in the topological $3$-sphere. Various numeric experiments are described, that corroborate our conjectures.

\medskip

\textbf{Acknowledgements.} A.K. was supported by the Swiss National Science Foundation (SNSF project no.~P300P2-151316) and the Japan Society for the Promotion of Science (Invitation Programs for Research project no.~S-14021). A.K. is thankful to Waseda University for hospitality during his visit. J.M. was supported by Waseda University (Grant for Special Research Projects no.~2014A-345) and the Japan Society for the Promotion of Science (Grant-in-Aid projects no.~25287014, no.~25610022).

\section{Preliminaries}\label{section:preliminaries}

Below, we describe a method to compute the volume of a simple polyhedron $P \subset \mathbb{H}^3$ that admits a decomposition into generalised hyperbolic tetrahedra. We recall, that a convex polytope $P$ is \textit{simple} if the valence of each vertex equals the dimension of $P$. Here, it means that each vertex of $P$ is trivalent. Now we define a generalised tetrahedron, and the decomposition of a given polyhedron $P$ into such generalised tetrahedra following a simple combinatorial procedure. 

Prior to doing so, we describe two combinatorial operations on the one-skeleton of $P$, associated with a decomposition into generalised tetrahedra. The first is the classical $I-H$ move, and the second is capping a triangular face with a tetrahedron:

\begin{figure}[h]
\begin{center}
\includegraphics* [scale=0.4]{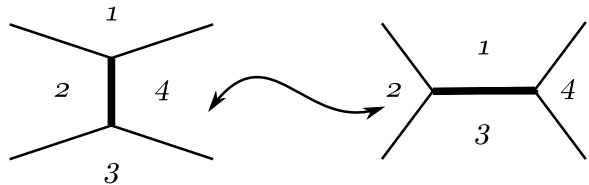}
\end{center}
\caption{An $I-H$ move}\label{fig:ih-move}
\end{figure}

\begin{itemize}
\item[1.] $I-H$ move: depicted in Fig.~\ref{fig:ih-move}. It consist of ``switching'' the edge between face $1$ and face $3$ to the edge between face $2$ and face $4$, and can be applied unless it collapses any of these faces.
\item[2.] Capping: depicted in Fig.~\ref{fig:capping-move}. It results in placing a tetrahedron over face $1$, such that the faces $2$, $3$ and $4$ continue until they reach the new vertex. 
\end{itemize}

\begin{figure}[h]
\begin{center}
\includegraphics* [scale=0.4]{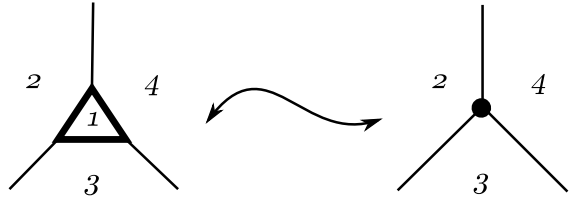}
\end{center}
\caption{A ``capping'' move}\label{fig:capping-move}
\end{figure}

These moves have been introduced in \cite{KR}, and later used in \cite{CGV, CM} for the purposes of defining a statistical sum over a spin network associated with a hyperbolic polyhedron. 

It's not hard to see that the following statement holds.

\begin{figure}[h]
\begin{center}
\includegraphics* [scale=0.4]{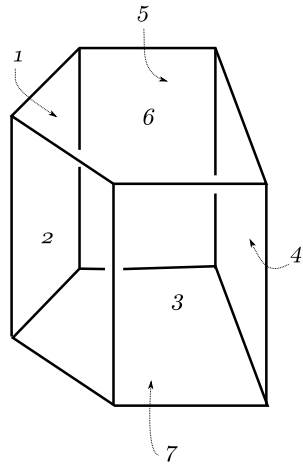}
\end{center}
\caption{A pentagonal prism $\Pi$}\label{fig:prism}
\end{figure}

\begin{lemma}\label{lemma:moves}
Suppose that $P$ is a simple polyhedron, which is not a tetrahedron. Then there exists a sequence of $I-H$ moves and capping moves transforming $P$ into a tetrahedron. 
\end{lemma}

\begin{example}
Let us take a pentagonal prism $\Pi$ depicted in Fig.~\ref{fig:prism} and apply a sequence of combinatorial transformations to its one-skeleton, as shown in Fig.~\ref{fig:prism-moves}. 

\begin{figure}[h]
\begin{center}
\includegraphics* [scale=0.33]{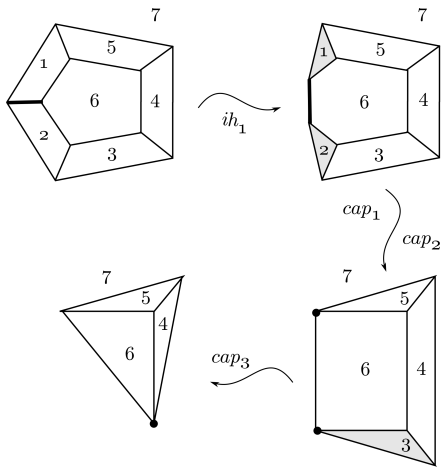}
\end{center}
\caption{A sequence of $I-H$ and capping moves applied to the prism $\Pi$}\label{fig:prism-moves}
\end{figure}

Thus, having applied a short sequence of combinatorial moves we obtain a tetrahedron. 
\end{example}

\begin{figure}[h]
\begin{center}
\includegraphics* [scale=0.35]{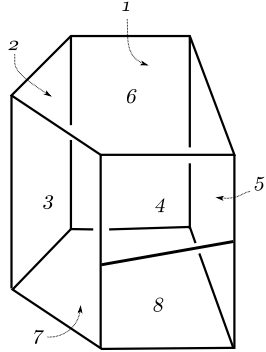}
\end{center}
\caption{A pentagonal pleated prism $\mathscr{P}$}\label{fig:pleated-prism}
\end{figure}

\begin{example}
Let us take a pentagonal pleated prism $\mathscr{P}$ depicted in Fig.~\ref{fig:pleated-prism} and apply a sequence of combinatorial transformation to its one-skeleton, as shown in Fig.~\ref{fig:pleated-prism-moves}. 
\end{example}

\begin{figure}[h]
\begin{center}
\includegraphics* [scale=0.33]{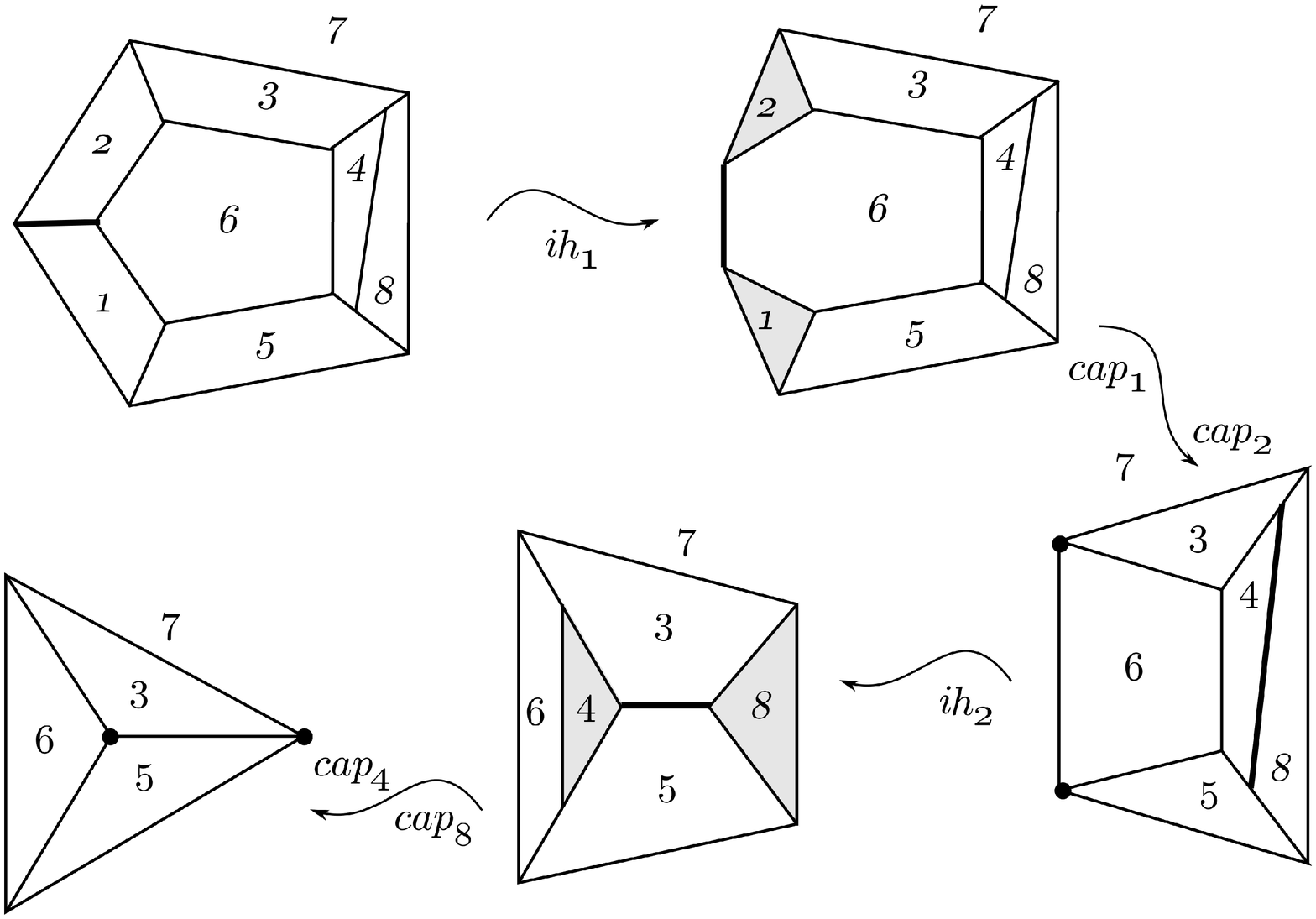}
\end{center}
\caption{A sequence of $I-H$ and capping moves applied to the pleated prism $\mathscr{P}$}\label{fig:pleated-prism-moves}
\end{figure}

Let us define a generalised hyperbolic tetrahedron $T$. Let $H_{n_k}$, $k=\overline{1,4}$, be four planes with the respective outer normal vectors $n_k$, $k=\overline{1,4}$. We compose the generalised Gram matrix of $T$ as the Gram matrix of the system of vectors $n_k$. From the matrix $G$ we produce the dual (generalised) Gram matrix of $T$, that corresponds to its ``vertices'' $v_k$, although some of them are outer normals to the respective polar planes $P_{v_k}$. If $v_k$ is ultra-ideal, we associate with it a plane $H_{v_k} = P_{v_k}$, otherwise we put $H_{v_k} = \mathbb{H}^3$. Given a plane $H_n$ with the outer normal $n$, let $H^{+}_n$ denote the half-space which $n$ points to and let $H^{-}_n$ denote the half-space which $-n$ points to. Then, the generalised hyperbolic tetrahedron $T$ is defined as
\begin{equation*}
T = \left( \bigcap^{4}_{k=1} H^{-}_{n_k} \right)\,\, \cap \,\, \left( \bigcap^{4}_{k=1} H^{-}_{v_k} \right).
\end{equation*}
Thus, each generalised tetrahedron $T$ still can be characterised by six parameters $a_1$, $a_2$, $\dots$, $a_6$, each of which corresponds to an edge of the ``standard'' tetrahedron. However, we have to distinguish two cases: 
\begin{itemize}
\item[1.] if the planes $H_{n_i}$ and $H_{n_j}$ intersect along an edge $e_{ij}$ with parameter $a_k$, then we set $a_k = e^{i \alpha_k}$, where $\alpha_k$ is the dihedral angle along $e_{ij}$;
\item[2.] if the planes $H_{n_i}$ and $H_{n_j}$ are ultra-parallel, then the vertices $v_k$ and $v_l$ are ultra-ideal and the planes $H_{v_k}$ and $H_{v_l}$ intersect along the common perpendicular $p_{ij}$ to $H_{n_i}$ and $H_{n_j}$: thus we set the respective parameter $a_k = e^{-\ell_k}$, where $\ell_k$ is the length of $p_{ij}$.
\end{itemize}
In Fig.~\ref{fig:prism-truncated-tetr} the case of a prism-truncated tetrahedron (see \cite{KM2012, KM-err} and \cite{KM2014}) is illustrated in detail. This type of generalised tetrahedron will appear quite often in the examples below, as well as several other types that we enumerate below. 

\begin{figure}[h]
\begin{center}
\includegraphics* [scale=0.33]{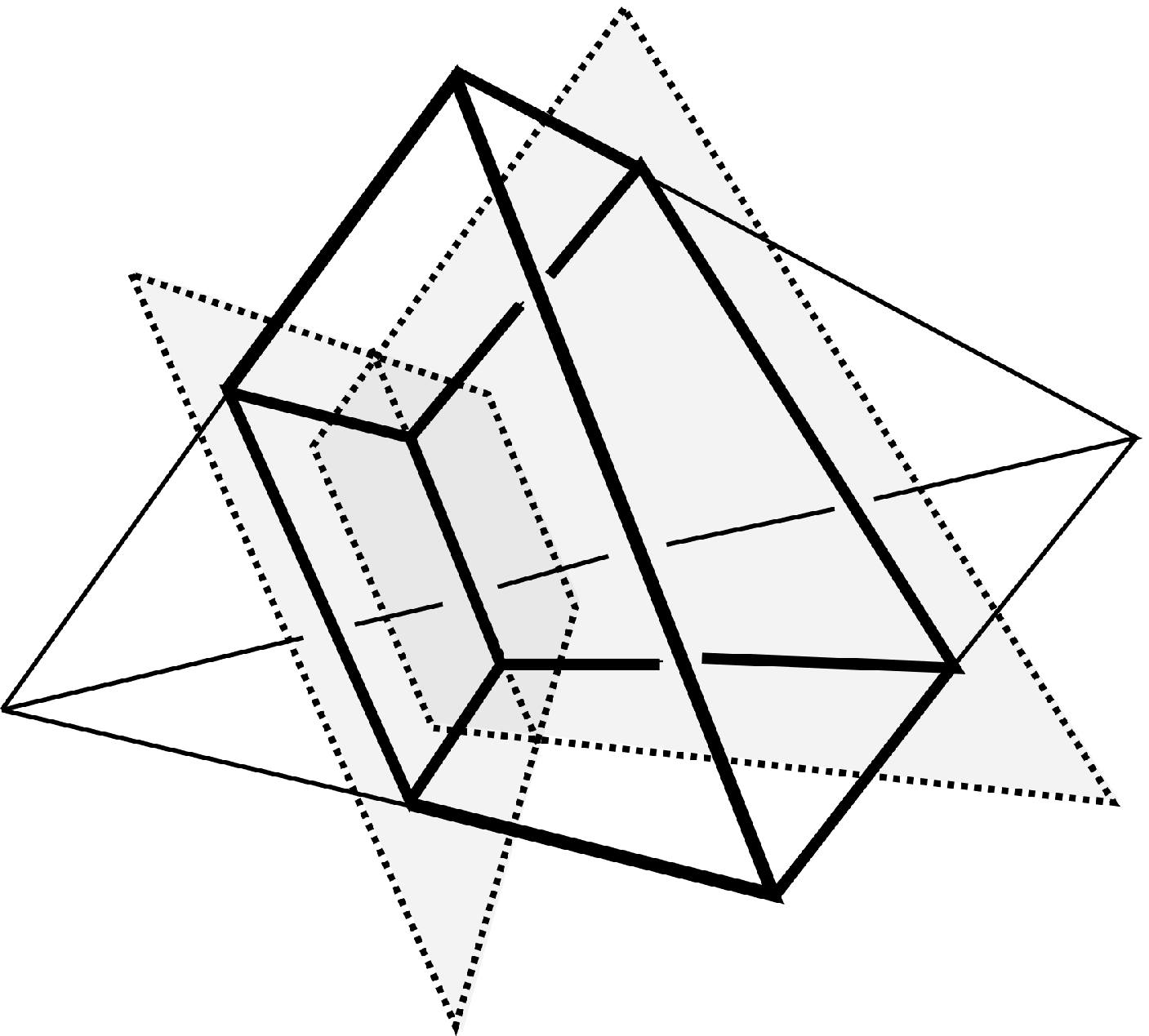}
\end{center}
\caption{A prism-truncated tetrahedron}\label{fig:prism-truncated-tetr}
\end{figure}

With each type of generalised tetrahedron $T$ we associate a symbol $t|\mathrm{a}\dots|\mathrm{p}\dots$ where we indicate which parameters $a_k$ correspond to dihedral angles (their indices follow the letter ``a'' in the symbol), and which parameters $a_k$ correspond to common perpendiculars' lengths (their indices follow the letter ``p''). Thus, given a symbol $t|\mathrm{a}\dots|\mathrm{p}\dots$, we can produce the corresponding generalised tetrahedron $T$ from the standard one by truncating the respective vertices and letting the corresponding polar planes intersect. Then we can recover the geometric parameters of $T$, such as its dihedral angles and/or edge lengths. The Gram matrix $G$ of $T$ in this case is still given by
\begin{equation}\label{eq:matrixGram}
G = \left( \begin{array}{cccc}
1 & -\frac{a_1 + 1/a_1}{2} & -\frac{a_2 + 1/a_2}{2} & -\frac{a_6 + 1/a_6}{2}\\
-\frac{a_1 + 1/a_1}{2} & 1 & -\frac{a_3 + 1/a_3}{2} & -\frac{a_5 + 1/a_5}{2}\\
-\frac{a_2 + 1/a_2}{2} & -\frac{a_3 + 1/a_3}{2} & 1 & -\frac{a_4 + 1/a_4}{2}\\
-\frac{a_6 + 1/a_6}{2} & -\frac{a_5 + 1/a_5}{2} & -\frac{a_4 + 1/a_4}{2} & 1
\end{array} \right),
\end{equation}
as if $T$ were a ``usual'' tetrahedron: it is still sufficient to determine $T$ uniquely, up to an isometry (see \cite{KM2012, KM-err} and \cite{Jacquemet}). 

Given a generalised hyperbolic tetrahedron $T$ with Gram matrix \eqref{eq:matrixGram}, we define the following auxiliary quantities in order to state a formula for $\mathrm{Vol}\, T$. 

Let $\mathscr{U} = \mathscr{U}(a_1,a_2,a_3,a_4,a_5,a_6,z)$ denote the function
\begin{eqnarray}\label{eq:UDefinition}
\lefteqn{\mathscr{U} = \mathrm{Li}_2(z) + \mathrm{Li}_2(a_1a_2a_4a_5z) + \mathrm{Li}_2(a_1a_3a_4a_6z) + \mathrm{Li}_2(a_2a_3a_5a_6z)}\\
\nonumber && - \mathrm{Li}_2(-a_1a_2a_3z) - \mathrm{Li}_2(-a_1a_5a_6z) - \mathrm{Li}_2(-a_2a_4a_6z) - \mathrm{Li}_2(-a_3a_4a_5z)
\end{eqnarray}
depending on seven complex variables $a_k$, $k=\overline{1,6}$ and $z$, where $\mathrm{Li}_2(\circ)$ is the dilogarithm function.

Let $z_{-}$ and $z_{+}$ be two solutions to the equation $e^{z \frac{\partial \mathscr{U}}{\partial z}} = 1$ in the variable $z$. According to \cite{MY}, these are
\begin{equation}\label{eq:Zpm}
z_{-} = \frac{-q_1-\sqrt{q^2_1-4q_0q_2}}{2q_2} \,\,\,\mbox{ and }\,\,\, z_{+} = \frac{-q_1+\sqrt{q^2_1-4q_0q_2}}{2q_2},
\end{equation}
where
\begin{equation*}
q_0 = 1 + a_1a_2a_3 + a_1a_5a_6 + a_2a_4a_6 + a_3a_4a_5
+ a_1a_2a_4a_5 + a_1a_3a_4a_6 + a_2a_3a_5a_6,
\end{equation*}
\begin{eqnarray}\label{eq:Q1Q2Q3Definition}
\nonumber q_1 = -a_1 a_2 a_3 a_4 a_5 a_6 \bigg{(}\bigg{(}a_1-\frac{1}{a_1}\bigg{)}\bigg{(}a_4-\frac{1}{a_4}\bigg{)} + \bigg{(}a_2-\frac{1}{a_2}\bigg{)}\bigg{(}a_5-\frac{1}{a_5}\bigg{)}\\
+\bigg{(}a_3-\frac{1}{a_3}\bigg{)}\bigg{(}a_6-\frac{1}{a_6}\bigg{)}\bigg{)},
\end{eqnarray}
\begin{eqnarray*}
q_2 = a_1 a_2 a_3 a_4 a_5 a_6 (a_1 a_4 + a_2 a_5 + a_3 a_6 + a_1a_2a_6 + a_1a_3a_5 + a_2a_3a_4 + \\a_4a_5a_6
+ a_1a_2a_3a_4a_5a_6).
\end{eqnarray*}

Given a function $f(x,y,\dots,z)$, let $f(x,y,\dots,z)\mid^{z=z_{-}}_{z=z_{+}}$ denote the difference $f(x,y,\dots,z_{-}) - f(x,y,\dots,z_{+})$. Now we define the following 
function $\mathscr{V} = \mathscr{V}(a_1,a_2,a_3,a_4,a_5,a_6,z)$ by means of the equality
\begin{equation}\label{eq:VDefinition}
\mathscr{V} = \frac{i}{4}\left( \mathscr{U}(a_1,a_2,a_3,a_4,a_5,a_6,z) - z\, \frac{\partial \mathscr{U}}{\partial z}\, \log z \right)\bigg{\vert}^{z=z_{-}}_{z=z_{+}}.
\end{equation}

Depending on the type $t|\mathrm{a}\dots|\mathrm{p}\dots$ of the tetrahedron $T$, we consider the following cases. 

\textit{1. One pair of ultra-parallel faces.} This is a prism truncated tetrahedron $T$, that has the symbol $t|\mathrm{a}ijklm|\mathrm{p}n$ for $i,j,k,l,m,n\in\{1,2,3,4,5\}$. In Fig.~\ref{fig:tetr1}, we illustrate the case $t|\mathrm{a}12356|\mathrm{p}4$. 

\begin{figure}[h]
\begin{center}
\includegraphics* [scale=0.33]{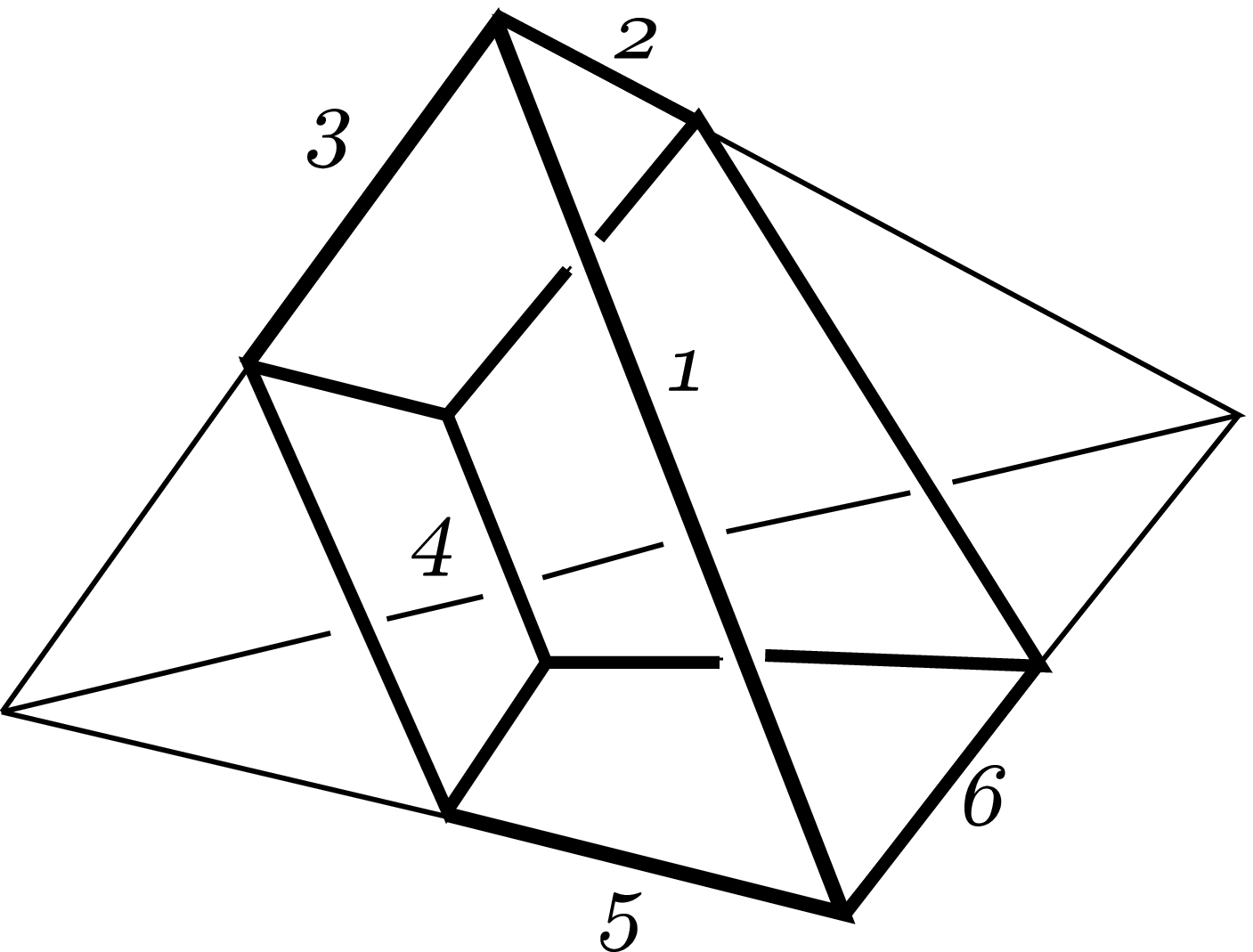}
\end{center}
\caption{A tetrahedron of type $t|\mathrm{a}12356|\mathrm{p}4$}\label{fig:tetr1}
\end{figure}

The volume of $T$ can be expressed by the formula
\begin{equation}\label{eq:vol1}
\mathrm{Vol}\,T = \Re \left( -\mathscr{V} + a_4\, \frac{\partial \mathscr{V}}{\partial a_4} \log a_4 \right),
\end{equation}
where $\mathscr{V} = \mathscr{V}(a_1,\dots, a_6)$ is the ``volume function'', as described in \cite{KM2014}. 

\textit{2. Two pairs of ultra-parallel faces.} In Fig.~\ref{fig:tetr2}, illustrate the case $t|\mathrm{a}2356|\mathrm{p}14$. Here, all the vertices of $T$ are ultra-ideal. 

\begin{figure}[h]
\begin{center}
\includegraphics* [scale=0.33]{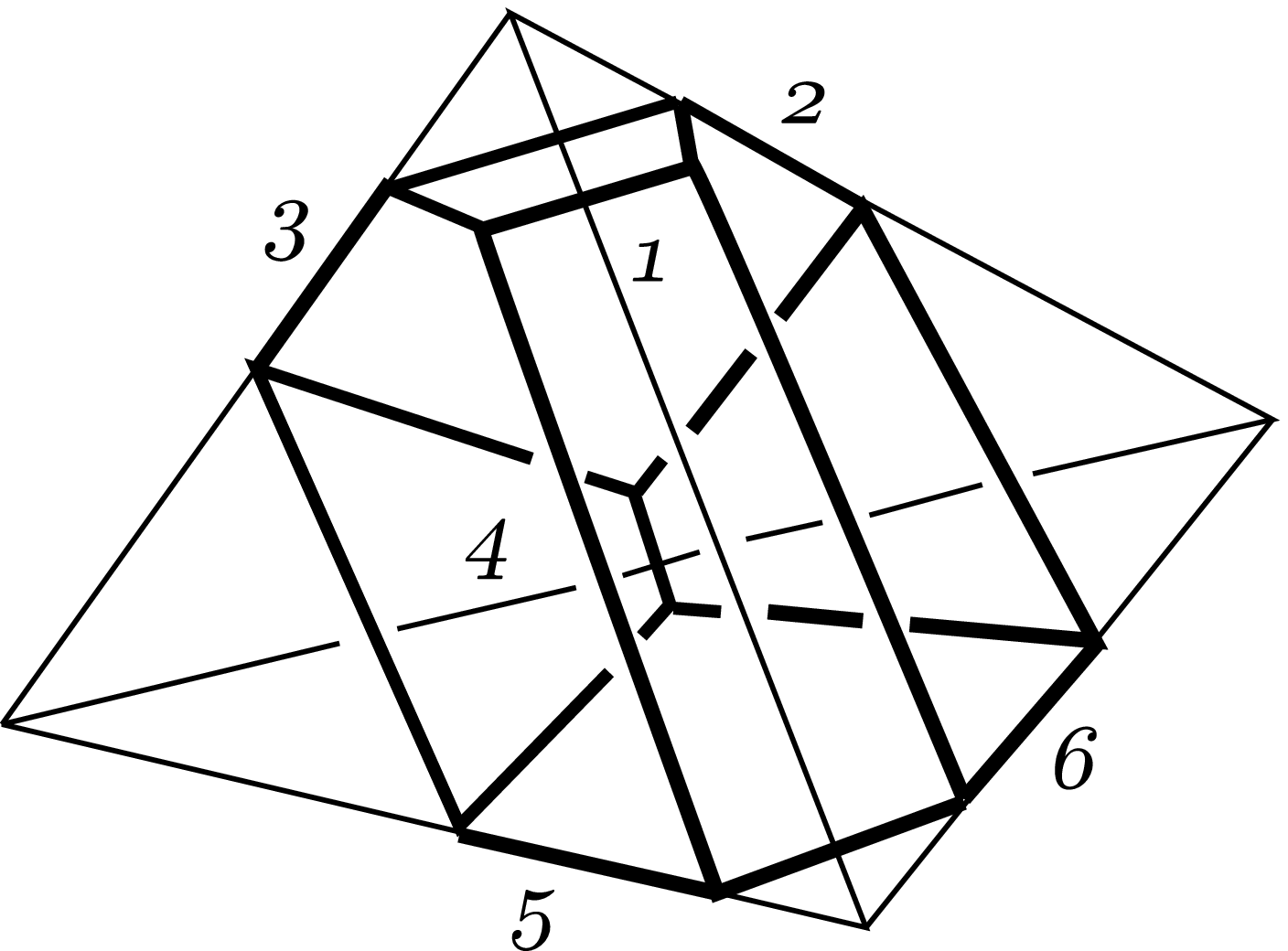}
\end{center}
\caption{A tetrahedron of type $t|\mathrm{a}2356|\mathrm{p}14$}\label{fig:tetr2}
\end{figure}

In complete analogy to the proof of \cite[Proposition 3]{KM2014}, we can show that the volume of $T$ in this case equals
\begin{equation}\label{eq:vol2}
\mathrm{Vol}\,T = \Re \left( -\mathscr{V} + \sum_{k\in \{1, 4\}} a_k\, \frac{\partial \mathscr{V}}{\partial a_k} \log a_k \right).
\end{equation}

\textit{3. Two pairs of ultra-parallel faces, another configuration.} A tetrahedron of type $t|\mathrm{a}1234|\mathrm{p}56$ is depicted in Fig.~\ref{fig:tetr3}. Here, three vertices $v_1$, $v_2$ and $v_3$ are ultra-ideal and the vertex $v_4$ is proper. 

\begin{figure}[h]
\begin{center}
\includegraphics* [scale=0.33]{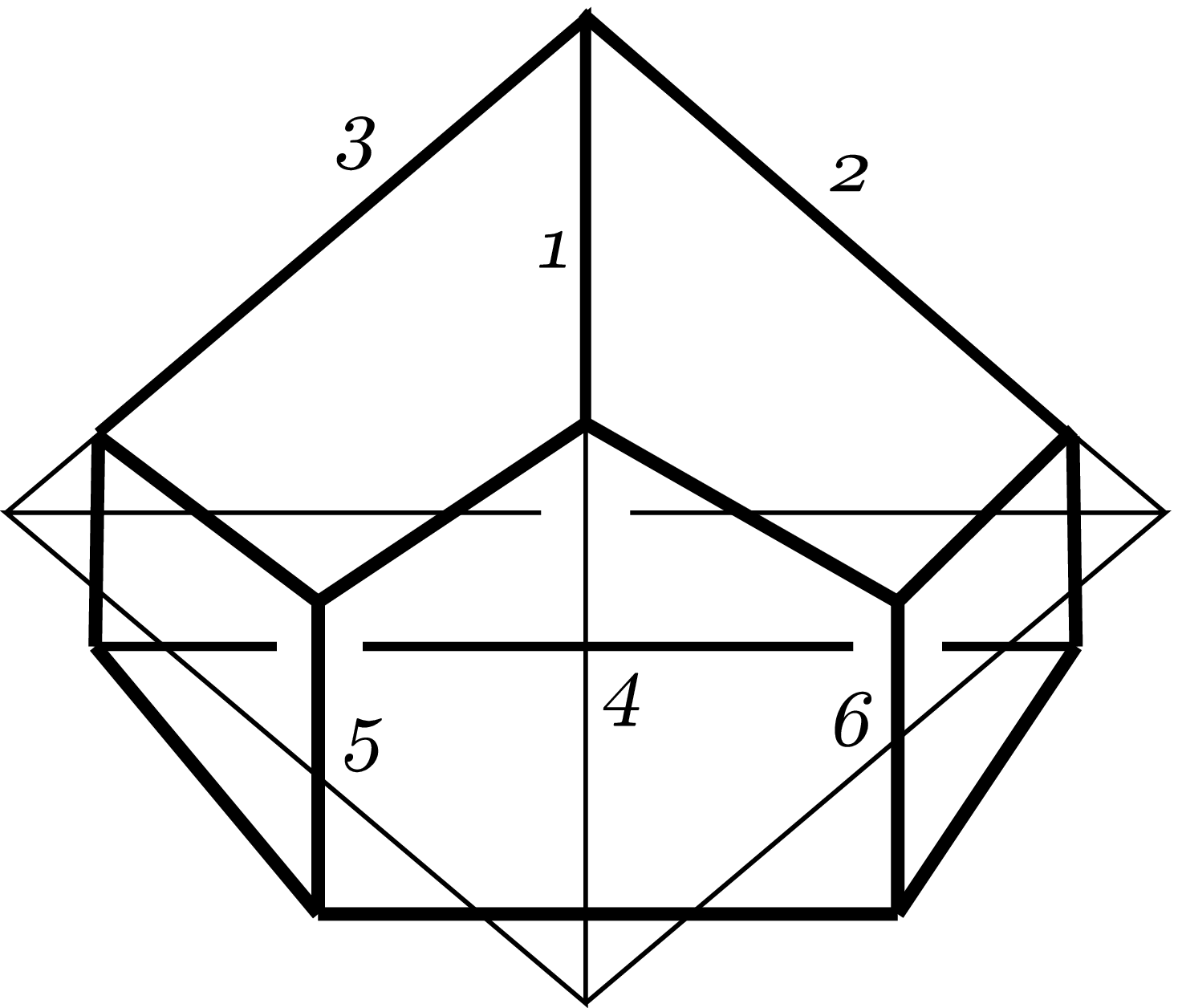}
\end{center}
\caption{A tetrahedron of type $t|\mathrm{a}1234|\mathrm{p}56$}\label{fig:tetr3}
\end{figure}

The volume of $T$ is
\begin{equation}\label{eq:vol3}
\mathrm{Vol}\,T = \Re \left( -\mathscr{V} + \sum_{k\in \{5, 6\}} a_k\, \frac{\partial \mathscr{V}}{\partial a_k} \log a_k \right).
\end{equation}

\begin{figure}[h]
\begin{center}
\includegraphics* [scale=0.33]{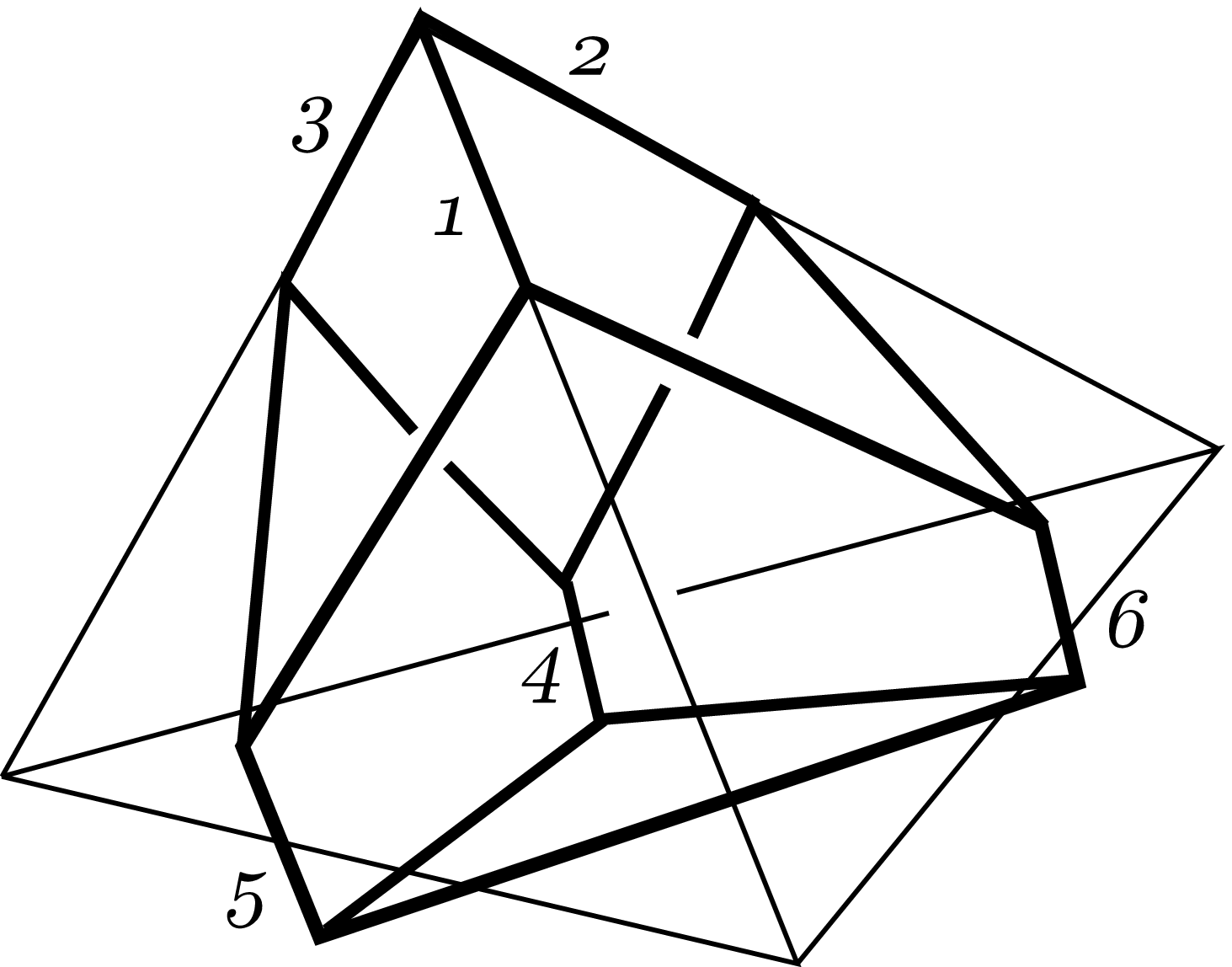}
\end{center}
\caption{A tetrahedron of type $t|\mathrm{a}123|\mathrm{p}456$}\label{fig:tetr4}
\end{figure}

\textit{4. One face that is ultra-parallel to three others.} In Fig.~\ref{fig:tetr4}, a tetrahedron of type $t|\mathrm{a}123|\mathrm{p}456$ is shown. Here, three faces are adjacent around the proper vertex $v_4$, while three other vertices $v_1$, $v_2$ and $v_3$ are ultra-ideal. 

The volume of $T$ is
\begin{equation}\label{eq:vol4}
\mathrm{Vol}\,T = \Re \left( -\mathscr{V} + \sum_{k\in \{4, 5, 6\}} a_k\, \frac{\partial \mathscr{V}}{\partial a_k} \log a_k \right).
\end{equation}

\begin{figure}[h]
\begin{center}
\includegraphics* [scale=0.33]{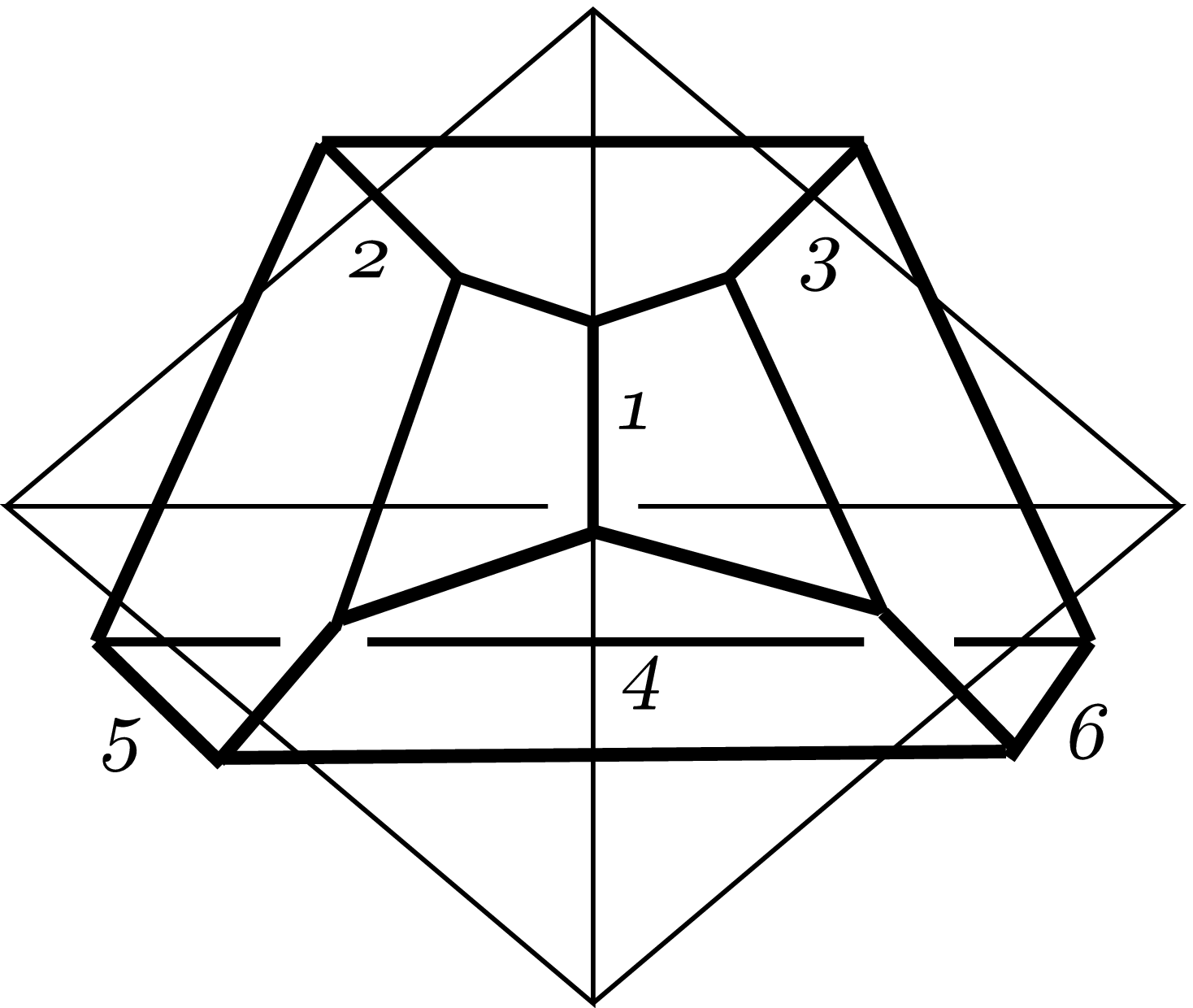}
\end{center}
\caption{A tetrahedron of type $t|\mathrm{a}14|\mathrm{p}2356$}\label{fig:tetr5}
\end{figure}

\textit{5. Four pairs of ultra-parallel faces.} We illustrate the case $t|\mathrm{a}14|\mathrm{p}2356$ in Fig.~\ref{fig:tetr5}. Here, $T$ has only ultra-ideal vertices and its faces are arranged in two groups: the faces $1$ and $2$ intersect along an edge, the faces $3$ and $4$ intersect along an edge, and each face in the pair $\{1, 2\}$ is ultra-parallel to any face in the pair $\{3, 4\}$.

The volume of $T$ equals
 \begin{equation}\label{eq:vol5}
\mathrm{Vol}\,T = \Re \left( -\mathscr{V} + \sum_{k\in \{2, 3, 5, 6\}} a_k\, \frac{\partial \mathscr{V}}{\partial a_k} \log a_k \right).
\end{equation}

\begin{lemma}
The volume of a generalised hyperbolic tetrahedron can be expressed by one of the formulae \eqref{eq:vol1}-\eqref{eq:vol5}.
\end{lemma}

The proof of the lemma above is analogous to the proof of \cite[Proposition 3]{KM2014} (each of the above formulae are simplified versions of \cite[Theorem 1]{KM2012}). We also refer the reader to \cite{ChoKim} for a geometric interpretation of the extra terms involving derivatives of $\mathscr{V}$ in the volume formulae above.

Let us now reconsider the previous examples of a pentagonal prism $\Pi$ and a pleated prism $\mathscr{P}$. We shall derive a geometric decomposition for each of them from the associated sequence of combinatorial transformations.

\begin{example}
Let us consider the first transformation $ih_1$ of the prism $\Pi$, as depicted in Fig.~\ref{fig:prism-moves}. Let us suppose that the common perpendicular $p_{67}$ to the plane $6$ and the plane $7$ is situated entirely inside $\Pi$. Then we draw two planes $S_1$ and $S_2$ orthogonal to the sides $1$ and $2$ of $\Pi$. Suppose these planes land on the sides of $\Pi$ in its interior. Then we cut along $\Pi$ along the planes $1$ and $2$, and thus ``chop off'' a polyhedron $T_1$ bounded by the planes of the faces $1$, $2$, $6$ and $7$ intersecting along the respective edges (or, more precisely, parts of edges) of the initial prism $\Pi$, and the planes $S_1$ and $S_2$ intersecting along $p_{67}$. Let us note that $T_1$ is nothing but a generalised tetrahedron of type $t|\mathrm{a}12345|\mathrm{p}6$. 

The capping transformations $cap_1$, $cap_2$ and $cap_3$ in Fig.~\ref{fig:prism-moves} correspond to further ``chopping off'' of generalised tetrahedra (of type $t|\mathrm{a}12345|\mathrm{p}6$): 
\begin{itemize}
\item the tetrahedron $T_2$, formed by the planes $2$, $3$, $6$, $7$, and the planes $S_2$, $S_3$ (orthogonal to the faces $2$ and $3$, respectively),
\item the tetrahedron the tetrahedron $T_3$, formed by the planes $1$, $5$, $6$, $7$, and the planes $S_3$, $S_5$ (orthogonal to the faces $1$ and $5$, respectively),
\item the tetrahedron the tetrahedron $T_4$, formed by the planes $3$, $4$, $6$, $7$, and the planes $S_3$, $S_4$ (orthogonal to the faces $3$ and $4$, respectively).
\end{itemize} 

\begin{figure}[h]
\begin{center}
\includegraphics* [scale=0.33]{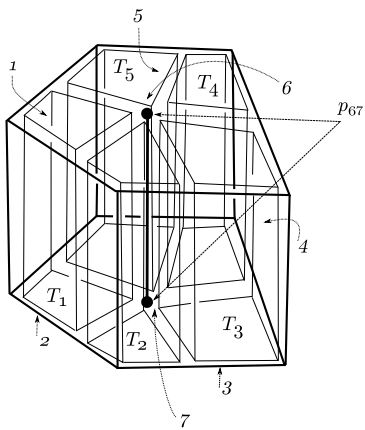}
\end{center}
\caption{A geometric decomposition of $\Pi$ into tetrahedra $T_i$, $i=1,\dots,5$, of type $t|\mathrm{a}12345|\mathrm{p}6$ }\label{fig:prism-moves-geometric}
\end{figure}

In each case we assume that the plane $S_k$ lands orthogonally on the $k$-th face of $\Pi$ in its interior. Thus, the prism $\Pi$ is divided into generalised tetrahedra. 

The last tetrahedron in Fig.~\ref{fig:prism-moves} corresponds to the generalised tetrahedron $T_5$ formed by the planes $4$, $5$, $6$, $7$ and $S_4$, $S_5$. 

The geometric decomposition of the prism $\Pi$ following the combinatorial moves in Fig.~\ref{fig:prism-moves} is depicted in Fig.~\ref{fig:prism-moves-geometric}, with all the tetrahedra $T_i$ marked. 
\end{example}

\medskip

\begin{example}
Let us consider the first transformation $ih_1$ of the pleated prism $\mathscr{P}$. Let us suppose that the common perpendiculars to any pairs of faces lie inside $\mathscr{P}$, if they exist. Let us take the common perpendicular $p_{67}$ to the plane $6$ and the plane $7$. Then we perform the transformation $ih_1$ on the edge $e_{12}$ (the intersection of the planes $1$ and $2$). This results in detaching a generalised tetrahedron $T_1$, corresponding to the planes $1$, $2$, $6$ and $7$ (the normals to these planes determine $T_1$ completely, however there are two more polar planes bounding it), whose edges consist of the edges involved in the $ih_1$ move, and the edges coming from the intersection of the planes $S_1$ and $S_2$, orthogonal to the faces $1$ and $2$, and passing through the common perpendicular $p_{67}$ to the plane $6$ and the plane $7$. The planes $S_1$ and $S_2$ intersect along $p_{67}$ and land on the faces $1$ and $2$, respectively, as well as they intersect the planes $6$ and $7$, yielding seven new edges in total. 

Here and below we suppose that each plane $S_k$, orthogonal to the $k$-th face of $\mathscr{P}$, lands on it orthogonally and inside $\mathscr{P}$. 

The transformations $cap_1$ and $cap_2$, which result in capping the triangular faces in Fig.~\ref{fig:pleated-prism-moves} mean further detaching of the generalised tetrahedra $T_2$ and $T_3$, formed by the planes $2$, $3$, $6$, $7$ and $1$, $5$, $6$, $7$, respectively. 

Now we consider the transformation $ih_2$ performed on the edge $e_{48}$ (the intersection of planes $4$ and $8$). We take the common perpendicular $p_{35}$ to the planes $3$ and $5$, and again suppose that it lies inside the pleated prism $\mathscr{P}$. We draw two planes passing through $p_{35}$: $S_4$ orthogonal to the face $4$, and $S_8$ orthogonal to the face $8$. Now, we detach one more generalised tetrahedron $T_4$ formed by the planes $3$, $4$, $5$ and $8$. 

The transformation $c_3$ and $c_4$ correspond to detaching two more tetrahedra: $T_5$ formed by the planes $3$, $5$, $6$ and $8$, and $T_6$ formed by the planes $3$, $5$, $4$ and $7$.

Here we supposed that all possible common perpendiculars are situated inside $\Pi$, the planes coming through the common perpendiculars orthogonally to the faces of $\mathscr{P}$ land in the interior of the faces. This implies that the generalised tetrahedra $T_k$ do not overlap. In other words, we suppose that the tetrahedra $T_k$ form a decomposition of $\mathscr{P}$: we shall write $\mathscr{P} = \sqcup_k T_k$. 
\end{example}

Below we give examples where the above decomposition is realised for a concrete polyhedron, having the combinatorial type of a pentagonal prism or a pleated prism and prescribed dihedral angles. 

\begin{figure}[h]
\begin{center}
\includegraphics* [scale=0.33]{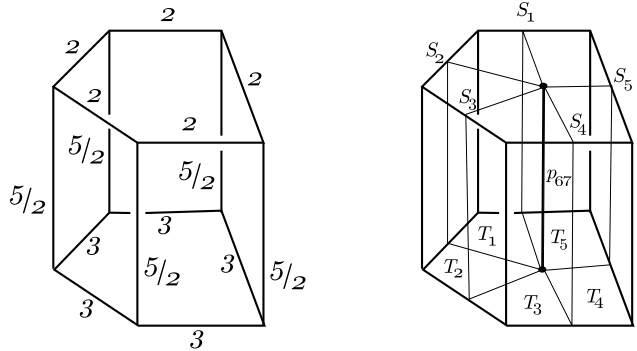}
\end{center}
\caption{A geometric realisation of the prism $\Pi$ with given dihedral angles, and its decomposition into generalised tetrahedra}\label{fig:prism-geometric}
\end{figure}

\begin{example}
Let us consider the pentagonal prism $\Pi$ depicted in Fig.~\ref{fig:prism-geometric}. Each edge of $\Pi$ is labelled with an label $q/p$, corresponding to a dihedral angle of $p\pi/q$. In this case, the common perpendicular $p_{67}$ to the planes $6$ and $7$ lies entirely inside the prism $\Pi$, and the planes $S_k$, $k=1,\dots,5$ passing through $p_{67}$ and orthogonal to the sides of $\Pi$ subdivide the prism into five ``prism truncated'' tetrahedra $T_k$ (i.e. tetrahedra of type $t|\mathrm{a}12345|\mathrm{p}6$), $k=\overline{1,5}$, which are mutually isometric. In this decomposition all the generalised tetrahedra $T_k$ have disjoint interiors, and $\Pi = \sqcup^{5}_{k=1} T_k$. The volume of each $T_k$, as well as the volume of the prism $\Pi$ is computed in \cite{KM2014}: $\mathrm{Vol}\,T_k = 0.52639$, $\mathrm{Vol}\,\Pi = 5 \cdot \mathrm{Vol}\,T_k \approx 2.63200$. Here we use formula \eqref{eq:vol1} to compute the volumes of all $T_k$, $k=1, \dots, 5$. 

As described in \cite{KM2014}, the volume computation proceeds by composing the function 
\begin{equation*}
\Phi(\ell) = \sum^{5}_{k=1} \mathrm{Vol}\,T_k + \pi \ell,
\end{equation*}
where $\ell$ is the length of the common perpendicular $p_{56}$, and solving the equation
\begin{equation*}
\frac{\partial \Phi(\ell)}{\partial \ell} = 0.
\end{equation*}

We can transform the last equation into an algebraic equation about $x = e^{-\ell}$:
\begin{equation*}
e^{\frac{\partial \Phi}{\partial \ell}} = 1.
\end{equation*}
This equation can be easily solved, which allows us to determine the necessary parameter $\ell$.
\end{example}

\begin{figure}[h]
\begin{center}
\includegraphics* [scale=0.33]{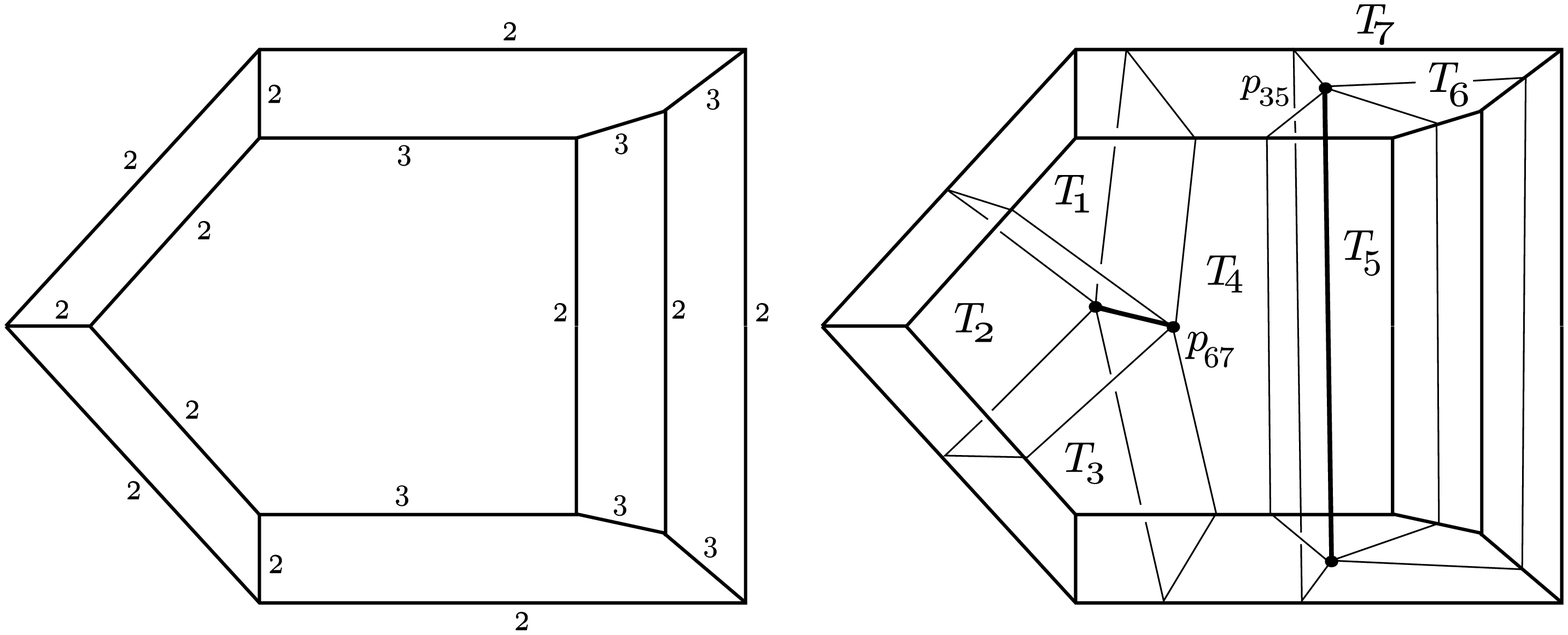}
\end{center}
\caption{A geometric realisation of the prism $\mathscr{P}$ with given dihedral angles, and its decomposition into generalised tetrahedra}\label{fig:pleated-prism-geometric}
\end{figure}

\begin{example}
Now we consider the pleated prism $\mathscr{P}$ depicted in Fig.~\ref{fig:pleated-prism-geometric} together with the labels for its dihedral angles. The common perpendiculars $p_1 = p_{67}$ (to the planes $6$ and $7$) and $p_2 = p_{35}$ (to the planes $3$ and $5$) in this case lie inside $\mathscr{P}$, and the planes $S_k$, $k=1,\dots, 8$, passing through each of them, respectively, land orthogonally into the interior of the $k$-th face. Then the pleated prism $\mathscr{P}$ is decomposed into seven generalised tetrahedra $T_k$: six are of type $t|\mathrm{a}12345|\mathrm{p}6$ (the tetrahedra $1$-$3$ and $5$-$7$), and one is of type $t|\mathrm{a}2356|\mathrm{p}14$ (the tetrahedron $4$ in the middle). 

We compute the volume of $\mathscr{P}$. In order to do so, let $\ell_k$, $k=1,2$, denote the length of $p_k$. Then we have
\begin{align*}
\mathrm{Vol}\,\mathscr{P} = \sum^{7}_{k=1} \mathrm{Vol}\,T_k,
\end{align*}
where we use formula \eqref{eq:vol1} for the tetrahedra $T_k$, $k\in \{1, 2, 3, 5, 6, 7\}$ (of type $t|\mathrm{a}12345|\mathrm{p}6$) and formula \eqref{eq:vol2} for the tetrahedron $T_4$ (of type $t|\mathrm{a}2356|\mathrm{p}14$). 

The tetrahedra $T_k$ form the pleated prism $\mathscr{P}$ when the angle sum around each $p_l$ equals $2\pi$.  The dihedral angle along the edge playing the role of $p_l$ in each generalised tetrahedron equals 
\begin{equation*}
\alpha_{p_l}(T_k) = 2 a_6 \frac{\partial \mathscr{V}}{\partial a_6}(T_k) \mod \pi,
\end{equation*}
with $l=1$ for the tetrahedra $T_k$, $k\in\{1, 2, 3\}$ and $l=2$ for the tetrahedra $T_k$, $k\in\{5, 6, 7\}$, while 
\begin{equation*}
\alpha_{p_1}(T_4) = 2 a_1 \frac{\partial \mathscr{V}}{\partial a_1}(T_4) \mod \pi,\hspace*{0.25in} \alpha_{p_2}(T_4) = 2 a_4 \frac{\partial \mathscr{V}}{\partial a_4}(T_4) \mod \pi.
\end{equation*}

Let us put
\begin{equation*}
\Phi(\ell_1, \ell_2) = \sum^{7}_{k=1} \mathrm{Vol}\,T_k + \pi \sum_{k\in\{1, 2\}} \ell_k. 
\end{equation*}

Then, from the above equations, we deduce that $\ell_k$ is a solution to the system of equations
\begin{equation*}
\frac{\partial \Phi}{\partial \ell_k} = 0,\,\,k=1,2.
\end{equation*}
We can transform the above equations into a system of algebraic equations in the variables $x_k = e^{-\ell_k}$:
\begin{equation*}
e^{\frac{\partial \Phi}{\partial \ell_k}} = 1,\,\,k=1,2.
\end{equation*}
This equation is as a polynomial equation in the variables $x_k$, and can be easily solved by computer. Moreover, the functions $\frac{\partial \Phi}{\partial \ell_k}$ are monotone in each variable $\ell_k$, $k=1,2$, while keeping the other variable fixed (in the domain where the function $\Phi$ is defined). This follows from the fact that the respective dihedral angles of the tetrahedra $T_k$, $k=1,\dots,7$, are monotone functions of $\ell_k$, $k=1,2$ (cf. Fig.~\ref{fig:angles-monotone}). 

\begin{figure}[h]
\begin{center}
\includegraphics* [scale=1]{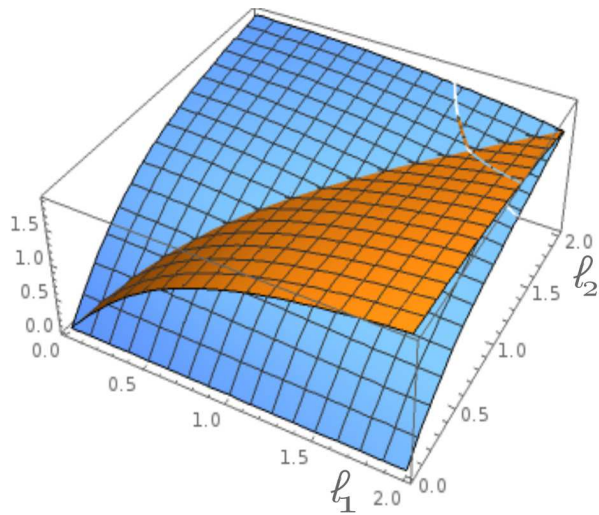}
\end{center}
\caption{The dihedral angles of the tetrahedron $T_4$ as fucntions of edge lengths $\ell_1$ and $\ell_2$ (the graph is symmetric under exchanging $\ell_1$ and $\ell_2$)}\label{fig:angles-monotone}
\end{figure}

Solving the above equations about $\ell_k$, we obtain that they have a unique common solution $(\ell_1, \ell_2) \approx (0.383438, 1.06239)$, and the volume of $\mathscr{P}$ is consequently equal to $\sim 2.34308$. Here we use formula \eqref{eq:vol1} to compute the volumes of $T_k$, $k\in \{1, 2, 3, 5, 6, 7\}$, and formula \eqref{eq:vol2} to compute the volume of $T_4$.
\end{example}

\section{Volume of a polyhedron}\label{section:volume-polyhedron}

We prove the following theorem, that allows us for computing the volume of a finite-volume simple polyhedron $P$, provided it has a decomposition into generalised tetrahedra, e.g. produced by a sequence of $IH$-moves and ``capping'' transformations. An analogous synthesis of a combinatorial decomposition followed by a geometric volume computation may be found in \cite{F, L} as applied to Euclidean polyhedra.

\begin{theorem}\label{thm:main}

Let $P \subset \mathbb{H}^3$ be a finite-volume simple polyhedron that admits a decomposition into generalised tetrahedra $T_k$ such that the volume of $P$ can be expressed as 
\begin{equation}\label{eq:volume}
\mathrm{Vol}\, P = \sum \mathrm{Vol}\,T_k(a_{k1}, \dots, a_{k6}), 
\end{equation}
where $a_{kl}$ are the six angle/length parameters corresponding to each generalised tetrahedron, the angle parameters $a_{kl} = e^{i \alpha_{kl}}$ come from the original dihedral angles of the polyhedron, and the edge parameters $a_{kl} = e^{- \ell_{kl}}$ are determined from a solution to the system of algebraic equations, in which each equation is associated with a common perpendicular $p_{kl}$ to some faces $F_k$ and $F_l$ of $P$ serving as an edge for a number of generalised tetrahedra $T_m$,
\begin{equation}\label{eq:glueing-1}
e^{\frac{\partial \Phi_{kl}(\ell_{kl})}{\partial \ell_{kl}}} = 1, 
\end{equation}
with
\begin{equation}\label{eq:glueing-2}
\Phi_{kl}(\ell_{kl}) = \sum_{p_{kl} \cap T_m \neq \emptyset} \mathrm{Vol}\,T_m(a_{m1}, \dots, a_{m6}) + \pi \sum \ell_{kl}, 
\end{equation}
where each monomial in the variables $a_{kl}$ has degree at most four. 
\end{theorem}

\begin{proof}
We have that 
\begin{equation}\label{eq:Schlaefli}
\frac{\partial \Phi_{kl}(\ell_{kl})}{\partial \ell_{kl}} = -\sum_{p_{kl} \cap T_m \neq \emptyset} \frac{\alpha_{kl}}{2} + \pi,
\end{equation}
where $\alpha_{kl}$ is the dihedral angle along an edge of a generalised tetrahedron $T_m$, corresponding to a common perpendicular $p_{kl}$ to some faces $F_k$ and $F_l$ of $P$. Since the polyhedron $P$ admits a decomposition into generalised tetrahedra $T_k$, then for some length parameters $\ell_{kl}$ (which will be the same for each tetrahedron $T_m$ whose edge $\ell_{kl}$ coincides with the common perpendicular $p_{kl}$) we have
\begin{equation}\label{eq:anglesum}
\sum_{p_{kl} \cap T_m \neq \emptyset} \alpha_{kl} = 2\pi.
\end{equation}
Thus, \eqref{eq:anglesum} implies \eqref{eq:Schlaefli} for these length $\ell_{kl}$. 

Looking at the expression for the function $\mathrm{Vol}\,T_k$, we observe that $\frac{\partial \Phi_{kl}(\ell_{kl})}{\partial \ell_{kl}}$ consists of logarithms of monomials in the parameters $a_{kl}$ where the longest monomial has the form $a_{k_1 l_1} a_{k_2 l_2} a_{k_3 l_3} a_{k_4 l_4}$. Thus, if all $a_{k_m k_n}$ happen to be edge parameters (i.e. $a_{k_m k_n} = e^{- \ell_{k_m k_n}}$) then we have an equation of degree four, at most. 
\end{proof}

Let us call a finite-volume polyhedron $P \subset \mathbb{H}^3$ ``width uniform'' if all the common perpendiculars to its facets lie inside $P$. The following fact obviously holds as a consequence of geodesic convexity.

\begin{lemma}
Every Coxeter polytope (or, more generally, every acute-angled polyhedron) is width uniform.  
\end{lemma}

A width uniform polyhedron has a geometric decomposition into generalised tetrahedra. We notice that each polytope $P$ may be decomposed into a number of generalised tetrahedra $T_k$, although some of them may overlap, and thus we shall need to use the ``inclusion-exclusion'' formula in order to express the volume of $P$ through the volumes of $T_k$. As well, some of the polar hyperplanes bounding $T_k$ may land on the faces of $P$ entirely or partially outside $P$ forming ``butterfly'' generalised tetrahedra, some of which are depicted in Fig.~\ref{fig:prism-butterfly}, c.f. \cite{KM2014}.

\begin{figure}[h]
\begin{center}
\includegraphics* [scale=0.33]{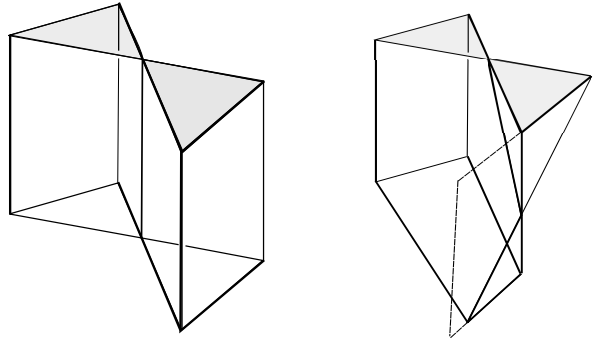}
\end{center}
\caption{Some examples of ``butterfly'' generalised tetrahedra}\label{fig:prism-butterfly}
\end{figure}

The volume formula of \cite{KM2014} works in either case, however, the glueing equations \eqref{eq:glueing-1}-\eqref{eq:glueing-2} will not have a unique solution. Still, the value of the volume of $P$ can be found by using \textit{some} solution among (usually few) solutions of the glueing equations. 

\section{Further computational examples}\label{section:compute-examples}

Below we give more numeric examples of volume computation for several polytopes. We start with the pleated prism $\mathscr{P}$, and then proceed to a dodecahedron $\mathscr{D}$. We shall use Coxeter dihedral angles (i.e. dihedral angles of the form $\pi/n$, for $n$ a natural number $\geq 2$), so we can ``double-check'' the volume of each polytope $P$ with the following procedure: the volume of $P$ equals the volume of the reflection orbifold obtained from $P$ by mirroring all its faces. We create a ``double'' of the respective reflection orbifold, which is an orientable orbifold $\mathscr{O}(P)$ with underlying space the three-sphere $\mathbb{S}^3$, and singular set combinatorially isomorphic to the $1$-skeleton of $P$. Each singular edge $e$ carries an angle of $2\pi/n$, corresponding to the angle of $\pi/n$ in the original polytope $P$. The volume of $\mathscr{O}(P)$ can be computed by the ``Orb'' software \cite{Orb}, and $\mathrm{Vol}\,\mathscr{O}(P) = 2\,\mathrm{Vol} P$. 

\begin{figure}[h]
\begin{center}
\includegraphics* [scale=0.33]{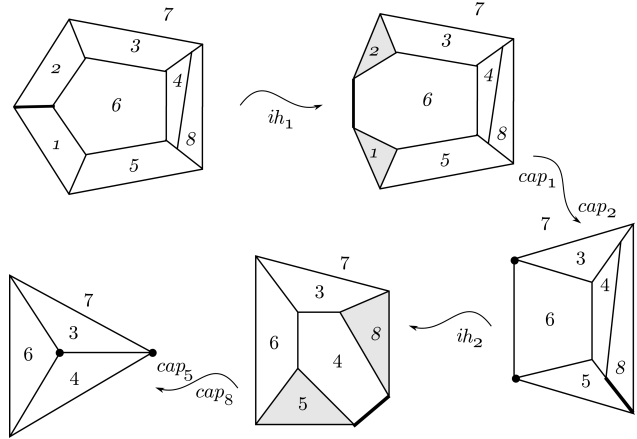}
\end{center}
\caption{Another sequence of $I-H$ and capping moves reducing the pleated prism $\mathscr{P}$ to a tetrahedron}\label{fig:pleated-prism-moves-2}
\end{figure}

First, we pick the pleated prism $\mathscr{P}$ from our previous example, and reduce it to a tetrahedron by a different sequence of $I-H$ and capping moves depicted in Fig.~\ref{fig:pleated-prism-moves-2}. Thus, the position of the common perpendiculars in the new subdivision will be different. 

\begin{figure}[h]
\begin{center}
\includegraphics* [scale=0.33]{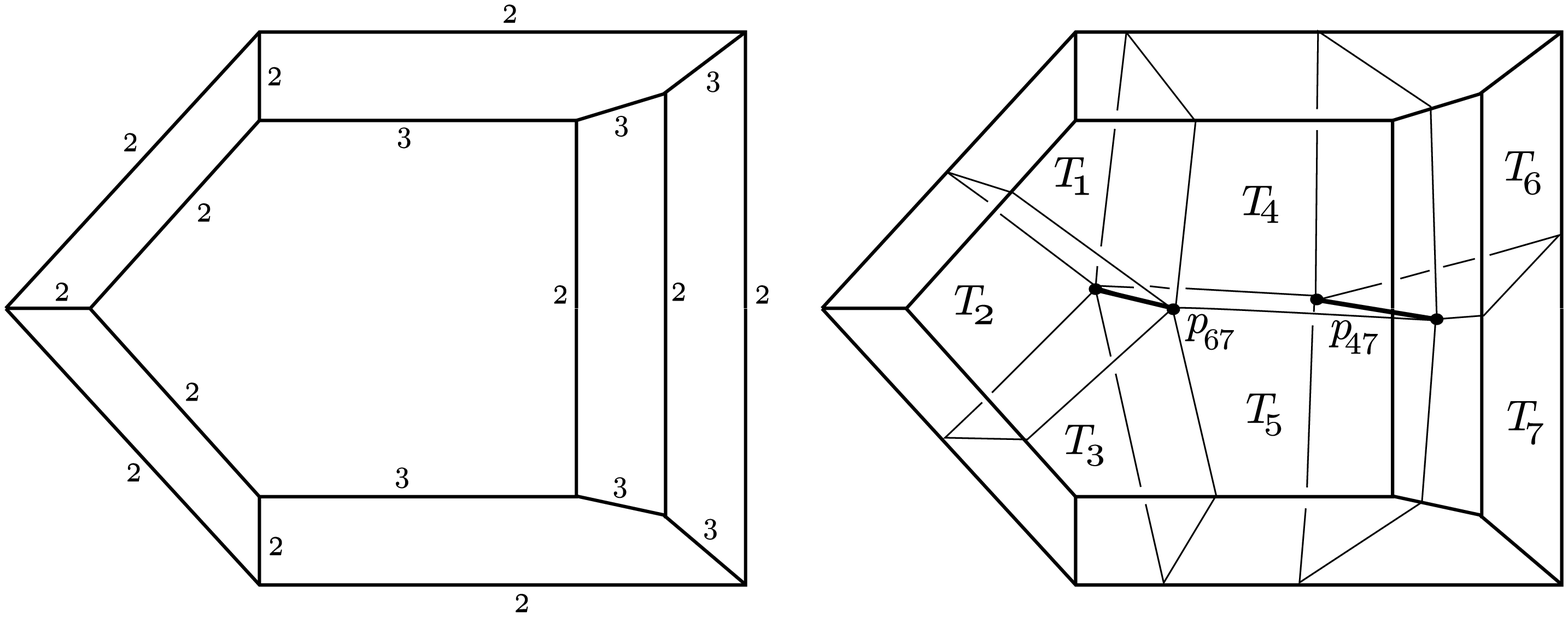}
\end{center}
\caption{A geometric realisation of the pleated prism $\mathscr{P}$ with given dihedral angles, and its decomposition into generalised tetrahedra}\label{fig:pleated-prism-geometric-2}
\end{figure}

As shown in Fig.~\ref{fig:pleated-prism-geometric-2}, we have again $7$ tetrahedra to consider, five of which are of the type $t|a12356|p4$ (the tetrahedra $T_i$, $i\in\{1,2,3,6,7\}$), and two are of the type $t|a1234|p56$ (the tetrahedra $T_i$, $i\in\{4,5\}$). Here, our computations find that the length of the common perpendicular $p_1 = p_{67}$ to the planes $6$ and $7$ is $\ell_1 \approx 0.383438$, and the length of the common perpendicular $p_2 = p_{47}$ to the planes $4$ and $7$ is $\ell_2 \approx 0.626516$. The volume of the prism $\mathscr{P}$ equals, as before, $\mathrm{Vol}\,\mathscr{P} \approx 2.34308$. 

By using various sequences of $I-H$ and capping moves, and therefore various geometric decompositions, we also found the volumes of pleated prisms depicted in Fig.~\ref{fig:pleated-prism-other}.

\begin{figure}[h]
\begin{center}
\includegraphics* [scale=0.33]{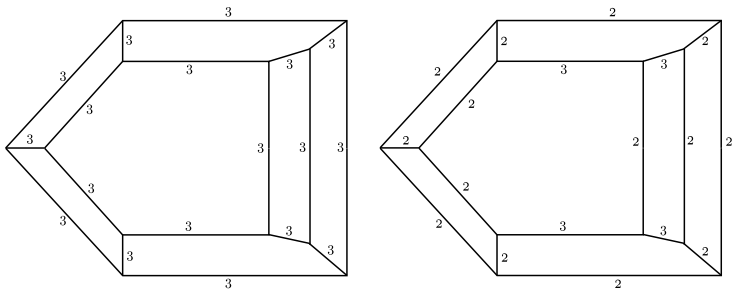}
\end{center}
\caption{Various pleated prisms with volumes $\sim 9.52855$ (left: this one has all ideal vertices) and $\sim 1.792925$ (right: it has many right angles and thus considerably smaller volume)}\label{fig:pleated-prism-other}
\end{figure}

Finally, we use our method to compute the volume of a hyperbolic dodecahedron $\mathscr{D}$. The sequence of $I-H$ and capping moves is presented in Fig.~\ref{fig:dodecahedron-moves-1}-Fig.~\ref{fig:dodecahedron-moves-3}, and leads to the following generalised tetrahedra in the geometric decomposition:
\begin{itemize}
\item $6$ generalised tetrahedra of type $t|\mathrm{a}12356\mathrm{p}4$,
\item $6$ generalised tetrahedra of type $t|\mathrm{a}1234\mathrm{p}56$,
\item $2$ generalised tetrahedra of type $t|\mathrm{a}123\mathrm{p}456$,
\item $2$ generalised tetrahedra of type $t|\mathrm{a}14\mathrm{p}2356$.
\end{itemize}

We consider three distinct dodecahedra: $\mathscr{D}_1$ with all right angles, $\mathscr{D}_2$ with all dihedral angles equal to $\pi/3$, and $\mathscr{D}_3$ with angles from the set $\{\pi/2, \pi/3, \pi/4\}$ distributed as depicted in Fig.~\ref{fig:dodecahedron-geometric}. The volumes of these dodecahedra equal respectively $\mathrm{Vol}\,\mathscr{D}_1 \approx 4.30621$, $\mathrm{Vol}\,\mathscr{D}_2 \approx 20.5802$, $\mathrm{Vol}\,\mathscr{D}_3 \approx 5.70085$. 

\afterpage{%

\begin{figure}[ht]
\begin{center}
\includegraphics* [scale=0.30]{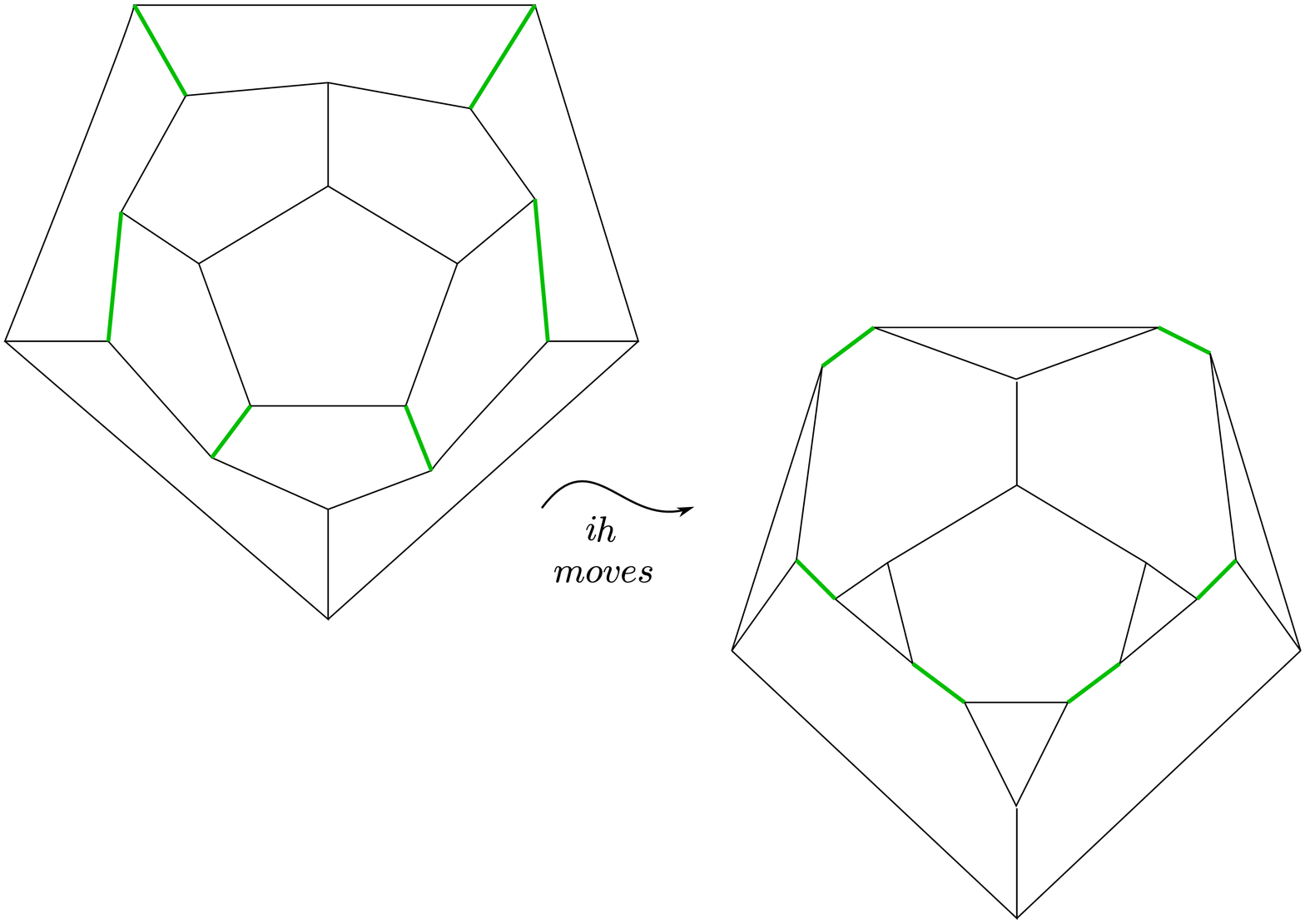}
\end{center}
\caption{The sequence of $I-H$ and capping moves for a dodecahedron}\label{fig:dodecahedron-moves-1}
\end{figure}

\begin{figure}[ht]
\begin{center}
\includegraphics* [scale=0.27]{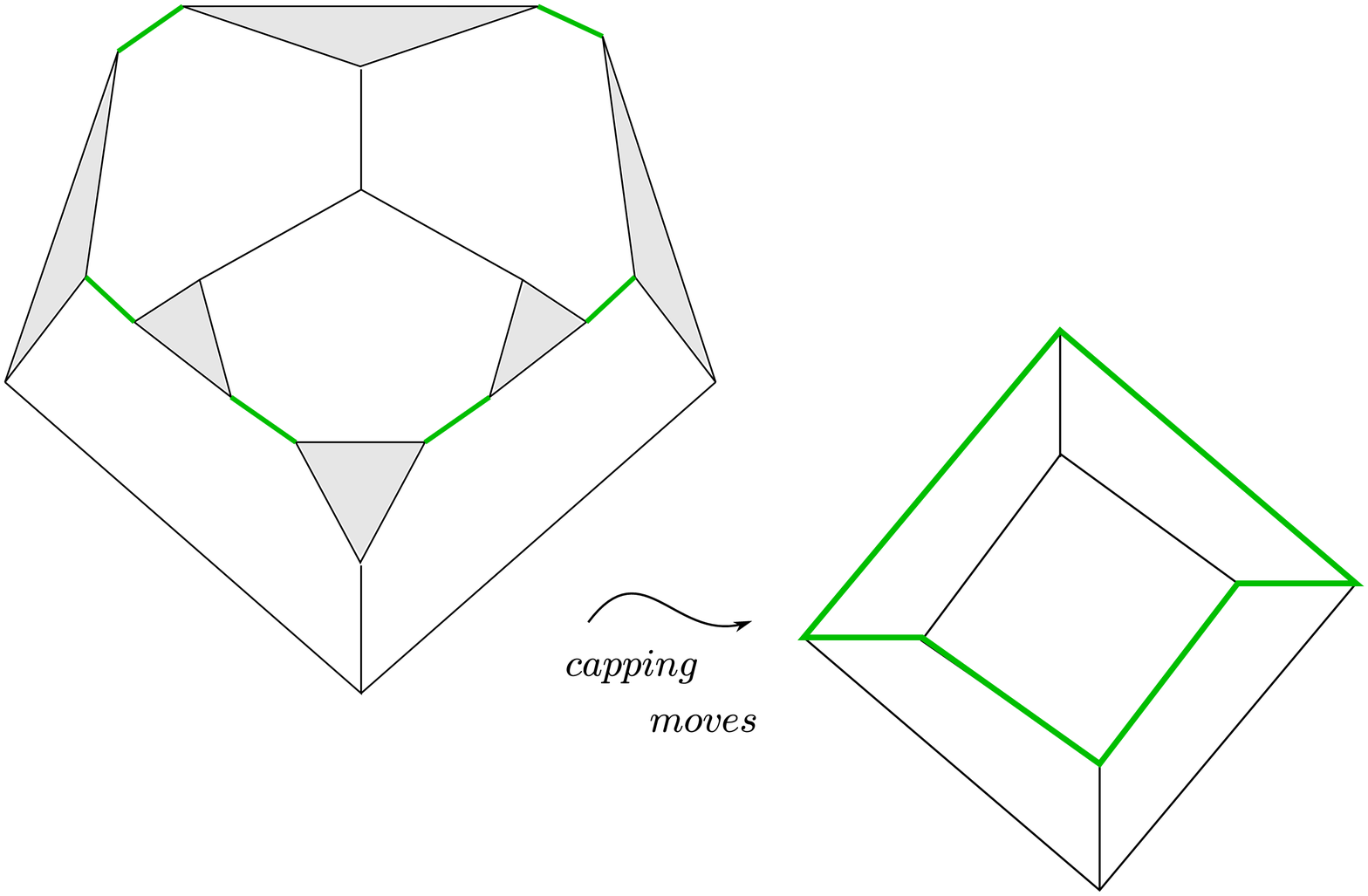}
\end{center}
\caption{The sequence of $I-H$ and capping moves for a dodecahedron}\label{fig:dodecahedron-moves-2}
\end{figure}

\clearpage
}

\afterpage{%

\begin{figure}[ht]
\begin{center}
\includegraphics* [scale=0.30]{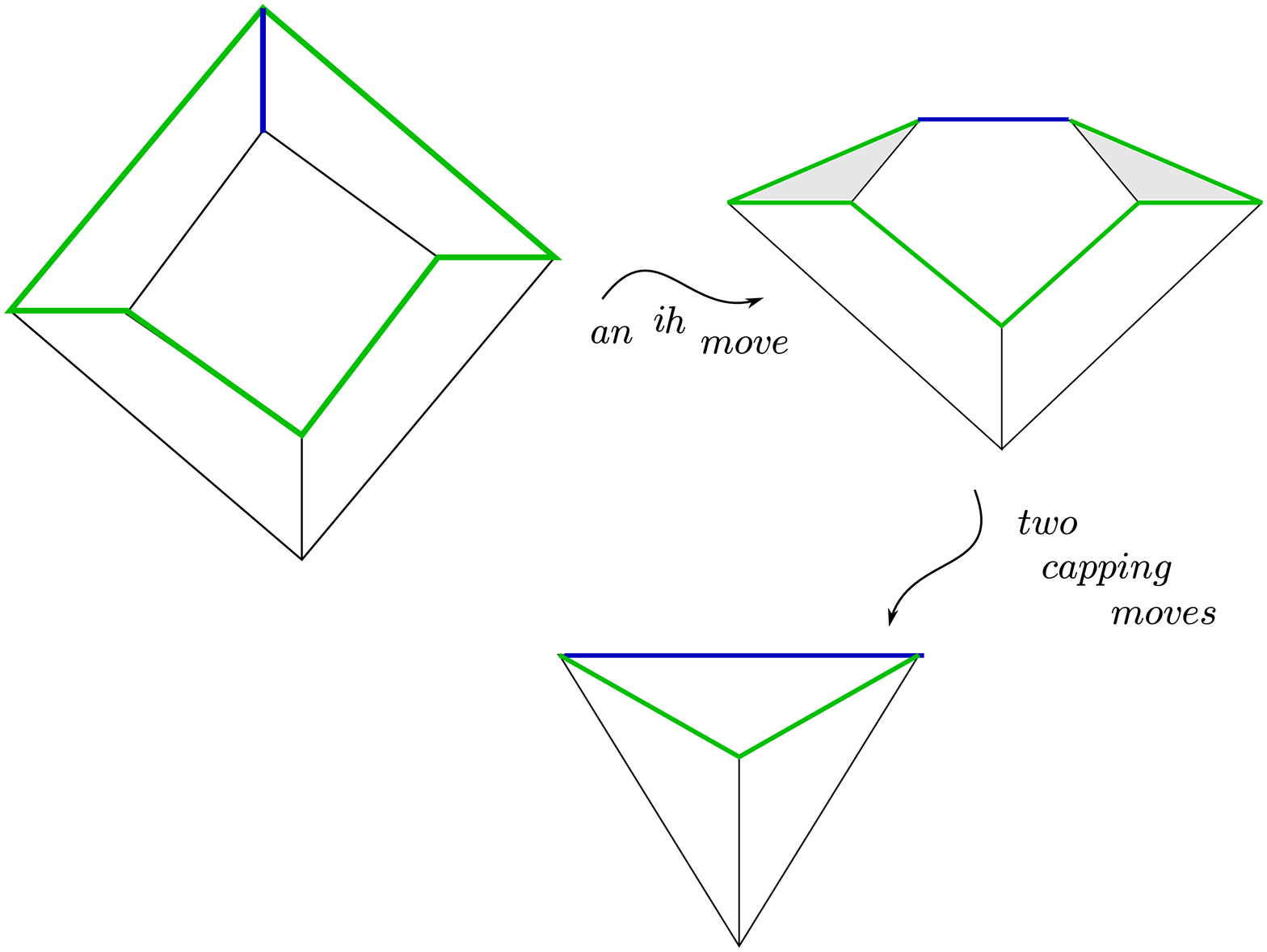}
\end{center}
\caption{The sequence of $I-H$ and capping moves for a dodecahedron}\label{fig:dodecahedron-moves-3}
\end{figure}

\begin{figure}[ht]
\begin{center}
\includegraphics* [scale=0.25]{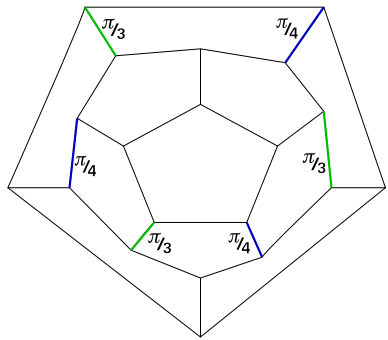}
\end{center}
\caption{A geometric realisation of the dodecahedron $\mathscr{D}_3$ with a rotational symmetry of order $3$. All unlabelled edges have dihedral angles $\pi/2$, the labelled ones have dihedral angles $\pi/3$ and $\pi/4$ in the alternating fashion.}\label{fig:dodecahedron-geometric}
\end{figure}

\clearpage
}

\section{Computing Kirillov-Reshetikhin invariants and the Volume Conjecture}\label{section:compute-conjecture}

We conjecture that the way in which a width-uniform polytope $P\subset \mathbb{H}^3$ is decomposed into generalised tetrahedra may be related to the sequence of combinatorial moves in which the Kirillov-Reshetikhin invariants of the corresponding trivalent graph $\Gamma$ (given by the $1$-skeleton of $P$) are computed. The following conjecture relating the volume $\mathrm{Vol}\, P$ to the asymptotic behaviour of the Kirillov-Reshetikhin invariants of $\Gamma$ has been corroborated by the results of our numerical experiments.

\begin{conj}\label{conj:KR-inv}
Let $P \subset \mathbb{H}^3$ be a hyperbolic polyhedron with edges $e_1$, $\cdots$, $e_m$ and dihedral angle along each edge $e_i$, $i=1,\dots,m$, respectively, $\alpha_i$. Let $\Gamma$ be the planar graph represented by the $1$-skeleton of $P$. Let  $c^{(r)}$ be a sequence of colourings $s_i^{(r)}$ $(1 \leq i \leq m, r = 3, 5, 7, \cdots)$ for the edges $e_1$, $\cdots$, $e_m$ of $\Gamma$ such that
$$
4\, \pi \, \lim_{n\to\infty}
\frac{s_i^{(2n+1)}}{2n+1} = \pi - \alpha_i,   
$$
then
$$
2\, \pi\,
\lim_{n\to\infty}
\frac{\log\langle\Gamma, c^{(2n+1)}\rangle}{2n+1} =
\mathrm{Vol}\, P,
$$
where $\langle \Gamma, c^{(2n+1)}\rangle$ is the unitary spin network (the Kirillov-Reshetikhin invariant with framing $c^{(2n+1)}$) associated with $\Gamma$, c.f. \cite{CM, CGV}.
\end{conj}

\begin{example}
Let $\Pi$ be the pentagonal prism depicted in Fig.~\ref{fig:prism-geometric}. Let $\Gamma$ be the planar graph given by the $1$-skeleton of $\Pi$, with edges $e_1$, $\dots$, $e_{15}$, and the corresponding dihedral angles $\alpha_1$, $\dots$, $\alpha_{15}$ along them. Here, we have that $\alpha_i \in \left\{ \frac{\pi}{2}, \frac{\pi}{3}, \frac{2\pi}{5} \right\}$. Let $c^{(r)}$ be the sequence of colourings $s_i^{(r)}$, $i=1,\dots,15$, $r=3,5,7,\dots$ of the edges $e_i$ of $\Gamma$ such that the conditions of Conjecture~\ref{conj:KR-inv} are satisfied. This means that we have $s_i^{(r)} = \frac{r-3}{4}$ for each $\alpha_i = \frac{\pi}{2}$, $s_i^{(r)} = \frac{r-3}{3}$ for each $\alpha_i = \frac{\pi}{3}$, and $s_i^{(r)} = \frac{3(r-3)}{10}$ for each $\alpha_i = \frac{2\pi}{5}$.

Then the unitary spin network associated with $\Gamma$ evaluates into
\begin{equation*}
\langle \Gamma, c^{(r)} \rangle = \sum_k (-1)^k \frac{\{k+1\}}{\{1\}} \left(\left\{
\begin{matrix}
\frac{r-3}{3}& \frac{r-3}{3}& \frac{3(r-3)}{10}\\ 
\frac{r-3}{4}& \frac{r-3}{4}& k
\end{matrix}
\right\}_{q=exp\left(\frac{4\pi i}{r}\right)}^{RW}\right)^5,
\end{equation*}
where $\left\{\begin{matrix}
a& b& e\\ 
d& c& f
\end{matrix}
\right\}_{q}^{RW}$ is the Racah-Wigner quantum $6j$-symbol, defined in \cite{KR} through the quantum numbers $\{n\}$ and quantum factorials $\{n\}!$.

In Table~\ref{tab:VC-prism} we collect the results of numerical computations of the quantity $V(r) = 2\pi\,\log \frac{\langle \Gamma, c^{(r)} \rangle}{r}$ for various odd values of $r$. 

\begin{table}[h]
\centering
\begin{tabular}{|c||c|c|c|c|c|}
\hline
$r$ & 483 & 963 & 1923 & 3843 & $\mathrm{Vol}\, \Pi$ 
\\
\hline
$V(r)$ & 2.27094 & 2.42388 & 2.51421 & 2.56627 & $\approx$ 2.63200
\\ \hline
\end{tabular}
\\[5pt]
\caption{The values of Kirillov-Reshetikhin invariants for $\Pi$}
\label{tab:VC-prism}
\end{table}

We also perform numerical computations for the pleated prism $\mathscr{P}$ depicted in Fig.~\ref{fig:pleated-prism-geometric}. The numeric values of its Kirillov-Reshetikhin invariants are given in Table~\ref{tab:VC-pleated-prism} for various odd values of $r$.

\begin{table}[h]
\centering
\begin{tabular}{|c||c|c|c|c|c|}
\hline
$r$ & 483 & 963 & 1923 & 3843 & $\mathrm{Vol}\, \mathscr{P}$ 
\\
\hline
$V(r)$ & 1.89647 & 2.08527 & 2.19702 & 2.26150 & $\approx$ 2.34308
\\ \hline
\end{tabular}
\\[5pt]
\caption{The values of Kirillov-Reshetikhin invariants for $\mathscr{P}$}
\label{tab:VC-pleated-prism}
\end{table}

Analogous, however way more tedious computations may be carried out in the case of the hyperbolic dodecahedra $\mathscr{D}_i$, $i=1,2,3$, described in Section~\ref{section:compute-examples}.
\end{example}

Each time an $I-H$ move or a capping move is performed on $\Gamma$, we have a generalised tetrahedron detached from the polyhedron $P$. As a step towards verification of Conjecture~\ref{conj:KR-inv} we suggest investigating if the following statements hold. 

\begin{conj}
Let $\alpha$, $\beta$, $\gamma$, $\delta$, $\epsilon$, $\phi$ be the dihedral angles of a generalised hyperbolic tetrahedron $T$, which may have ideal and ultra-ideal vertices, though no truncating polar planes determined by its ultra-ideal vertices intersect. Let $a^{(r)}$, $b^{(r)}$, $\cdots$, $f^{(r)}$, $r = 3, 5, 7, \cdots$ be sequences of integers satisfying
$$
4\, \pi\,\lim_{r\to\infty}\frac{a^{(r)}}{r} = \pi - \alpha, \quad
4\, \pi\,\lim_{r\to\infty}\frac{b^{(r)}}{r} = \pi - \beta, \quad
4\, \pi\,\lim_{r\to\infty}\frac{c^{(r)}}{r} = \pi - \gamma,
$$
$$
4\, \pi\,\lim_{r\to\infty}\frac{d^{(r)}}{r} = \pi - \delta, \quad
4\, \pi\,\lim_{r\to\infty}\frac{e^{(r)}}{r} = \pi -  \epsilon, \quad
4\, \pi\,\lim_{r\to\infty}\frac{f^{(r)}}{r} = \pi - \phi,
$$
where $r$ is always an odd integer. Then
\begin{equation}\label{eq:tetr-6j-1}
2\, \pi \, \lim_{r\to\infty}
\frac{1}{r}\, \log
\left\{\begin{matrix}
a^{(r)} & b^{(r)} & e^{(r)} \\
d^{(r)} & c^{(r)} & f^{(r)}
\end{matrix}\right\}_{q=\exp(\frac{4 \pi i}{r})}^{RW}
= {\rm Vol}(T).  
\end{equation}
\end{conj}

In Fig.~\ref{fig:tetr-6j-1} we illustrate the values of $V(\alpha) = \mathrm{Vol}\,T_\alpha$, for a regular tetrahedron with all dihedral angles equal to $\alpha$. The tetrahedron $T_\alpha$ is spherical for $\alpha > \arccos(1/3)$, Euclidean for $\alpha = \arccos(1/3)$, and hyperbolic for $\pi/3<\alpha < \arccos(1/3)$. It is ideal for $\alpha=\pi/3$, truncated (i.e. has ultra-ideal vertices) hyperbolic for $0<\alpha<\pi/3$, and it deforms into a regular ideal octahedron $O$ for $\alpha=0$. The octahedron $O$ has all dihedral angles equal to $\pi/2$  and its volume is $\mathrm{Vol}\, O \approx 3.663...$. 
 
\begin{figure}[h]
$$
\epsfig{file=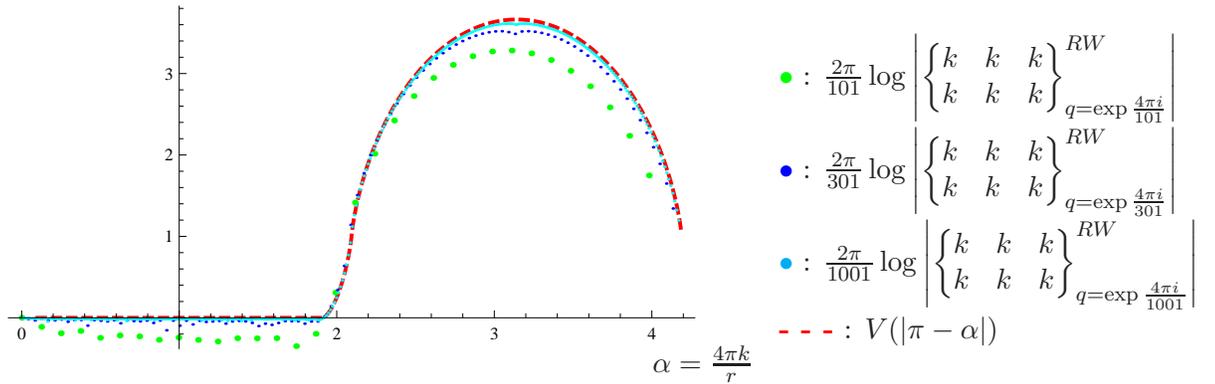, scale=1.0}
\raisebox{-3mm}{\hspace{-6mm}$\alpha = \frac{4 \pi k}{r}$}
\raisebox{20mm}{
\begin{tabular}{l}
\textcolor{green}{$\bullet$} : $\frac{2\pi}{101}\log\left|
\left\{
\begin{matrix}
k&k&k\\k&k&k
\end{matrix}
\right\}_{q=\exp \frac{4 \pi i}{101}}^{RW}
\right|$
\\[12pt]
\textcolor{blue}{$\bullet$} : $\frac{2\pi}{301}\log\left|
\left\{
\begin{matrix}
k&k&k\\k&k&k
\end{matrix}
\right\}_{q=\exp \frac{4 \pi i}{301}}^{RW}
\right|$
\\[12pt]
\textcolor{cyan}{$\bullet$} : $\frac{2\pi}{1001}\log\left|
\left\{
\begin{matrix}
k&k&k\\k&k&k
\end{matrix}
\right\}_{q=\exp \frac{4 \pi i}{1001}}^{RW}
\right|$
\\[12pt]
\textcolor{red}{{\bf - - -}} :  $V(|\pi - \alpha|)$
\end{tabular}}
$$
\caption{Values of the quantum $6j$-symbol and $V(\alpha)$.}
\label{fig:tetr-6j-1}
\end{figure}

Let us notice that the tetrahedron considered in the above conjecture is supposed to have only \textit{mild} vertex truncations, c.f. \cite{KM2012}. Another conjecture considers a generalised hyperbolic tetrahedron $T$ with some vertices \textit{intensely} truncated. We specify this conjecture for the case of a prism truncated tetrahedron, considered in \cite{KM2012, KM2014}. Although, one may state an analogous conjecture for a generalised tetrahedron of any type considered in Section~\ref{section:preliminaries}. 

\begin{conj}
Let $\alpha$, $\beta$, $\gamma$, $\delta$, $\epsilon$, $\phi$ be the dihedral angles of a prism truncated hyperbolic tetrahedron $T$, which may have ideal or ultra-ideal vertices.\footnote{However, only two truncating polar planes determined by its ultra-ideal vertices intersect.} Let the parameter $\phi$ represent the dihedral angle at the edge arising from intense truncation.  
Let $a^{(r)}$, $b^{(r)}$, $\cdots$, $f^{(r)}$ $(r = 3, 5, 7, \cdots)$ be sequences of integers satisfying
$$
4\, \pi\,\lim_{r\to\infty}\frac{a^{(r)}}{r} = \pi - \alpha, \quad
4\, \pi\,\lim_{r\to\infty}\frac{b^{(r)}}{r} = \pi - \beta, \quad
4\, \pi\,\lim_{r\to\infty}\frac{c^{(r)}}{r} = \pi - \gamma,
$$
$$
4\, \pi\,\lim_{r\to\infty}\frac{d^{(r)}}{r} = \pi - \delta, \quad
4\, \pi\,\lim_{r\to\infty}\frac{e^{(r)}}{r} = \pi -  \epsilon, \quad
4\, \pi\,\lim_{r\to\infty}\frac{f^{(r)}}{r} = \pi - \phi,
$$
where $r$ is always an odd integer. Then
\begin{multline}\label{eq:tetr-6j-2}
2\, \pi \, \lim_{r\to\infty}
\sum_{j\,\, odd}\frac{1}{r}\, \log
\left\{\begin{matrix}
a^{(r)} & b^{(r)} & e^{(r)} \\
d^{(r)} & c^{(r)} & \frac{j-1}{2}
\end{matrix}\right\}_{q=\exp(\frac{4 \pi i}{r})}^{RW}
\!\!\!\! \frac{\{(j+1) \, (2\, f^{(r)}+1)\}}{\{j\}}
\\
= {\rm Vol}(T).  
\end{multline}
\end{conj}

Let $T_{\alpha, \beta}$ be the doubly truncated tetrahedron with dihedral angle $\alpha$ at its usual five edges and $\beta$ at the edge arising from intense truncation. Let 
\begin{equation*}
U(k, l) = 
\sum_{j}\left\{
\begin{matrix}
k&k&k\\k&k&j
\end{matrix}
\right\}_{q=exp\left(\frac{4\pi i}{r}\right)}^{RW}
\, \frac{\{(2\,j+1)(2\,l+1)\}}{\{2j+1\}}.  
\end{equation*}
In Fig.~\ref{fig:tetr-6j-2} the values of $U((r-3)/3, l)$ and $V(\beta) = \mathrm{Vol}\, T_{\pi/3, \beta}$ are shown.  
\begin{figure}[htb]
$$
\epsfig{file=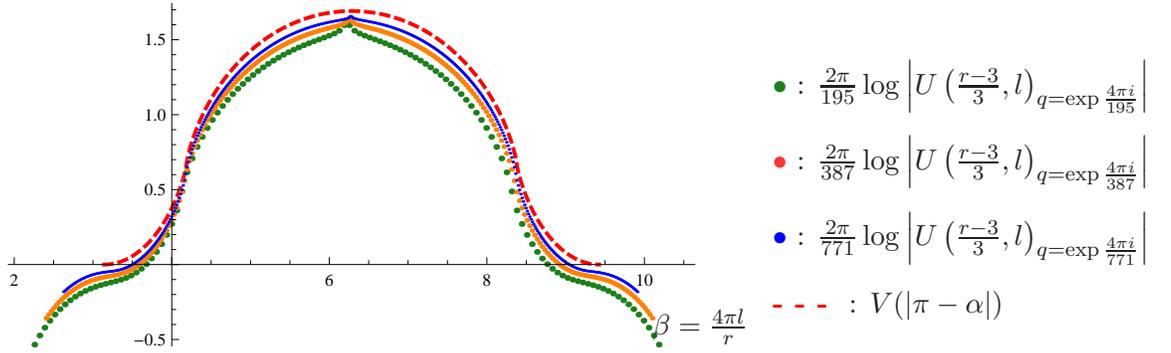, scale=1.0}
\raisebox{2mm}{\hspace{-6mm}$\beta = \frac{4 \pi l}{r}$}
\raisebox{20mm}{
\begin{tabular}{l}
\textcolor[rgb]{0.1,0.5,0.1}{$\bullet$} : $\frac{2\pi}{195}\log\left|
U\left(\frac{r-3}{3}, l\right)_{q=\exp \frac{4 \pi i}{195}}
\right|$
\\[12pt]
\textcolor[rgb]{1,0.2,0.2}{$\bullet$} : $\frac{2\pi}{387}\log\left|
U\left(\frac{r-3}{3}, l\right)_{q=\exp \frac{4 \pi i}{387}}
\right|$
\\[12pt]
\textcolor{blue}{$\bullet$} : $\frac{2\pi}{771}\log\left|
U\left(\frac{r-3}{3}, l\right)_{q=\exp \frac{4 \pi i}{771}}
\right|$
\\[12pt]
\textcolor{red}{{\bf - - - }} :  $V(|\pi - \alpha|)$
\end{tabular}}
$$
\caption{Values of $U(\frac{r-3}{3}, l)$ and $V(\beta)$.} 
\label{fig:tetr-6j-2}
\end{figure}

Since the volume of $P$ is a sum of volumes of the generalised tetrahedra $T_i$ composing it, we hope that there might be a correspondence between the asymptotic behaviour of the sums \eqref{eq:tetr-6j-1}-\eqref{eq:tetr-6j-2} associated with each tetrahedron $T_i$, and the asymptotic behaviour of the Kirillov-Reshetikhin invariant $\langle \Gamma, c^{(r)} \rangle$ of the whole polyhedron $P$.

\bigskip

\begin{flushleft}
\textit{
Alexander Kolpakov\\
Department of Mathematics\\
University of Toronto\\
40 St. George Street\\
Toronto ON\\
M5S 2E4 Canada\\
kolpakov.alexander(at)gmail.com}
\end{flushleft}

\medskip

\begin{flushleft}
\textit{
Jun Murakami\\
Department of Mathematics\\
Faculty of Science and Engineering\\
Waseda University\\
3-4-1 Okubo Shinjuku-ku\\ 
169-8555 Tokyo, Japan\\
murakami(at)waseda.jp}
\end{flushleft}


\begin{thebibliography}{}


\bibitem{ChoKim}\textsc{Y.~Cho, H.~Kim} {``Geometric and analytic interpretation of orthoscheme and Lambert cube in extended hyperbolic space,''} J. Korean Math. Soc. \textbf{50}~(6), 1223-1256~(2013).








\bibitem{CM} \textsc{F.~Costantino, J.~Murakami} {``On $SL(2, \mathbb{C})$ quantum $6j$-symbol and its relation to the hyperbolic volume''}, arXiv:1005.4277

\bibitem{CGV} \textsc{F.~Costantino, F.~Gu\'{e}ritaud, R. van der Veen} {``On the volume conjecture for polyhedra''}, arXiv:1403.2347

\bibitem{F} \textsc{P. Filliman} {``The volume of duals and sections of polytopes,''} Mathematika \textbf{39}~(1), 67-80 (1992).

\bibitem{Orb}\textsc{D.~Heard} {``Orb''}, An interactive software for finding hyperbolic structures on $3$-dimensional orbifolds and manifolds, available from http://www.ms.unimelb.edu.au/$\sim$snap/orb.html

\bibitem{Jacquemet}\textsc{M.~Jacquemet} {``The inradius of a hyperbolic truncated $n$-simplex''}, Discrete Comput. Geom. \textbf{51}, 997-1016 (2014).

\bibitem{Kashaev}\textsc{R.M.~Kashaev} {``The hyperbolic volume of knots from the quantum dilogarithm,''} Lett. Math. Phys. \textbf{39}~(3), 269-275 (1997); arXiv:q-alg/9601025.

\bibitem{KR}\textsc{A.N. Kirillov, N. Yu. Reshetikhin} {``Representations of the algebra $U_q(sl(2))$, $q$-orthogonal polynomials and invariants of links,''} in Infinite-dimensional Lie algebras and groups (Luminy-Marseille, 1988), pp. 285--339, Adv. Ser. Math. Phys., \textbf{7}, World Sci. Publ., Teaneck, NJ, 1989.


\bibitem{KM2012}\textsc{A.~Kolpakov, J.~Murakami} {``Volume of a doubly truncated hyperbolic tetrahedron,''} Aequationes Math., \textbf{85}~(3), 449-463 (2013); arXiv:1203.1061.

\bibitem{KM-err}\textsc{A.~Kolpakov, J.~Murakami} {``Erratum to: Volume of a doubly truncated hyperbolic tetrahedron,''} Aequationes Math., \textbf{88}~(1-2), 199-200 (2014).

\bibitem{KM2014}\textsc{A.~Kolpakov, J.~Murakami} {``The dual Jacobian of a generalised tetrahedron, and volumes of prisms''}, Tokyo J. Math., to appear; arXiv:1409.3355. 

\bibitem{W}\textsc{Wolfram Research}, ``Mathematica 9'', a computational software; http://www.wolfram.com/mathematica/







\bibitem{L} \textsc{J. Lawrence} {``Polytope volume computation,''} Mathematics of Computation \textbf{57}~(195), 259-271 (1991).

\bibitem{MY}\textsc{J.~Murakami, M.~Yano} {``On the volume of hyperbolic and spherical tetrahedron,''} Comm. Annal. Geom. \textbf{13}~(2), 379-400~(2005).

\end{thebibliography}
\end{document}